\documentclass{amsart}

\usepackage{latexsym,amsfonts} 
\usepackage{amsmath,amssymb,graphics,setspace}
\usepackage{amsthm}
\usepackage{MnSymbol}
\usepackage{mathrsfs} 
\usepackage{fontenc}%

\usepackage{paralist}
\usepackage{color}
\usepackage[verbose]{wrapfig}
\usepackage[ruled,vlined,linesnumbered]{algorithm2e}
\usepackage{marvosym}
\usepackage{array}
\usepackage{picins,graphicx}
\usepackage{pstricks,pst-text,pst-grad,pst-node,pst-3dplot,pstricks-add,pst-poly,pst-coil,pst-tree} 
\usepackage{multido,pst-node,pst-bspline}
\usepackage{multirow}
\usepackage{caption}
\usepackage{subfigure}
\usepackage{accents}
\renewcommand{\thesubfigure}{\thefigure.\arabic{subfigure}}
\makeatletter
\renewcommand{\p@subfigure}{}
\renewcommand{\@thesubfigure}{\thesubfigure:\hskip\subfiglabelskip}
\makeatother

\usepackage{hyperref}
\hypersetup{linktocpage=true,colorlinks=true,linkcolor=blue,citecolor=blue,pdfstartview={XYZ 1000 1000 1}}

\newcommand{\near}{\delta} 
\newcommand{\dnear}{\delta_{\Phi}} 
\newcommand{\dcap}{\mathop{\cap}\limits_{\Phi}} 
\newcommand{\assign}{\mathrel{\mathop :}=}

\newcommand{\sn}{\mathop{\delta}\limits^{\doublewedge}} 

\newcommand{\snd}{\mathop{\delta_{_{\Phi}}}\limits^{\doublewedge}} 
\newcommand{\eot}{\qquad \textcolor{blue}{\Squaresteel}}

\newcommand{\ponear}[1]{\mathop{\check{\delta}}\limits_{#1}}
\newcommand{\tonear}[1]{\mathop{\bar{\delta}}\limits_{#1}}
\newcommand{\cnear}[1]{\mathop{\mathring{\delta}}\limits_{#1}}

\newtheorem{example}{Example}

\newtheorem{definition}{Definition}
\newtheorem{lemma}{Lemma}
\newtheorem{theorem}{Theorem}

\usepackage[normalem]{ulem} 

\begin{document}

\title[Order Induced Proximity]{Proximity Induced by Order Relations}

\author[M.Z. Ahmad]{M.Z. Ahmad$^{\alpha}$}
\email{ahmadmz@myumanitoba.ca}
\address{\llap{$^{\alpha}$\,}
Computational Intelligence Laboratory,
University of Manitoba, WPG, MB, R3T 5V6, Canada}
\thanks{\llap{$^{\alpha}$\,}The research has been supported by University of Manitoba Graduate Fellowship and Gorden P. Osler Graduate Scholarship.}

\author[J.F. Peters]{J.F. Peters$^{\beta}$}
\email{James.Peters3@umanitoba.ca}
\address{\llap{$^{\beta}$\,}
Computational Intelligence Laboratory,
University of Manitoba, WPG, MB, R3T 5V6, Canada and
Department of Mathematics, Faculty of Arts and Sciences, Ad\.{i}yaman University, 02040 Ad\.{i}yaman, Turkey}
\thanks{\llap{$^{\beta}$\,}The research has been supported by the Natural Sciences \&
Engineering Research Council of Canada (NSERC) discovery grant 185986 
and Instituto Nazionale di Alta Matematica (INdAM) Francesco Severi, Gruppo Nazionale per le Strutture Algebriche, Geometriche e Loro Applicazioni grant 9 920160 000362, n.prot U 2016/000036.}


\date{}

\dedicatory{Ju. M. Smirnov and S.A. Naimpally}

\begin{abstract}
This paper introduces an order proximity on a collection of objects induced by a partial order using the Smirnov closeness measure on a Sz\'{a}z relator space.   A Sz\'{a}z relator is a nonempty family of relations defined on a nonvoid set $K$.   The Smirnov closeness measure provides a straightforward means of assembling partial ordered of pairwise close sets.  In its original form, Ju. M. Smirnov closeness measure $\delta(A,B) = 0$ for a pair of nonempty sets $A,B$ with nonvoid intersection and $\delta(A,B) = 1$ for non-close sets.  A main result in this paper is that the graph obtained by the proximity is equivalent to the Hasse diagram of the order relation that induces it.   This paper also includes an application of order proximity in detecting sequences of video frames that have order proximity.
\end{abstract}

\maketitle
\section{Introduction}\label{sec:intro}
This paper introduces proximities in terms of order relations on sets of objects with proximity induced by the order.
Axioms for the closeness (proximity) of nonempty sets have been given in terms of set intersections\cite{Cech1966}\cite{Lodato1964}\cite{Lodato1966}. Different proximities consider different associated sets such as set closures\cite{Wallman1938}, set interiors\cite{peters2015strongly} and set descriptions\cite{Peters2016CP}.
   
Traditionally, the closeness (proximity) of nonempty sets is viewed in terms of asymptotically close sets or those sets that have points in common.   With descriptive proximity, closeness of non-disjoint as well as disjoint nonempty sets occurs in cases where sets have matching descriptions.   A description of a set is a feature vector whose components are feature values that quantify the characteristics of a set.   Instead of the usual forms of proximity, the paper considers proximities relative to a partial ordering of elements of a set or collections of elements of sets or collections of nonempty sets or collections of cells in a cell complex or descriptions of collections of nonempty sets using a combination of a Sz\'{a}z relator~\cite{Szaz1987,Szaz2000,Szaz2014} on a nonempty set and the Smirnov closeness measure~\cite{Smirnov1952,Smirnov1952a}.  The main result in this paper states the equivalence of graphs obtained by proximity to the Hasse diagrams of order relations from which they are induced(see Theorem~\ref{thm:proxhassec}).


\section{Preliminaries}\label{sec:prelim}
Recall that a \emph{binary relation} is a subset of the Cartesian product,$X \times Y$. Then,
\begin{definition}\label{def:porder}
	A \emph{partial order} is a binary relation $\leq$, over a set $X$ satisfying for any $a,b,c \in X$:
	\begin{compactenum}[1$^o$]
		\item \textbf{Reflexivity:} $a \leq a$
		\item \textbf{Antisymmetry:} if $a \leq b$ and $b \leq a$, then $a=b$
		\item \textbf{Transitivity:} if $a \leq b$ and $b \leq c$, then $a \leq c$  
	\end{compactenum}
\end{definition}
Let us now define a total order.
\begin{definition}\label{def:torder}
	If a partial order $\leq$ also satisfies,
	\begin{description}
		\item[4$^o$] \textbf{Connex property or Totality:} either $a \leq b$ or $b \leq a$
	\end{description}
	then, $\leq$ is termed a \emph{total order}. Hence, the total order is a special case of a partial order.
\end{definition}	 
Let us now consider a pair $(X, \leq)$, where $X$ is a set and $\leq$ is a partial(total) order, then $(X,\leq)$ is a \emph{partially(totally) ordered set}. 
\begin{figure}
	\centering
	\begin{subfigure}[Total order on $\mathbb{Z}$]
		{\includegraphics[width=1.3in, bb=130 480 310 670,clip]{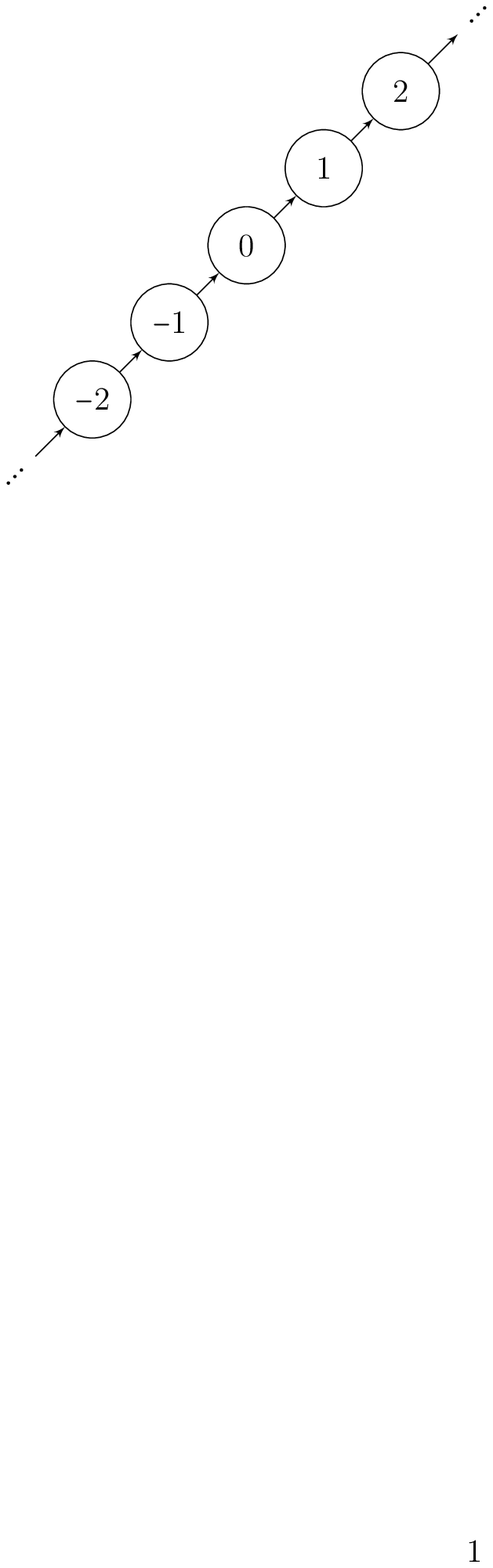}
			\label{subfig:integers}}
	\end{subfigure}
	\begin{subfigure}[Partial order on a power set]
		{\includegraphics[width=1.3in, bb=130 550 250 670,clip]{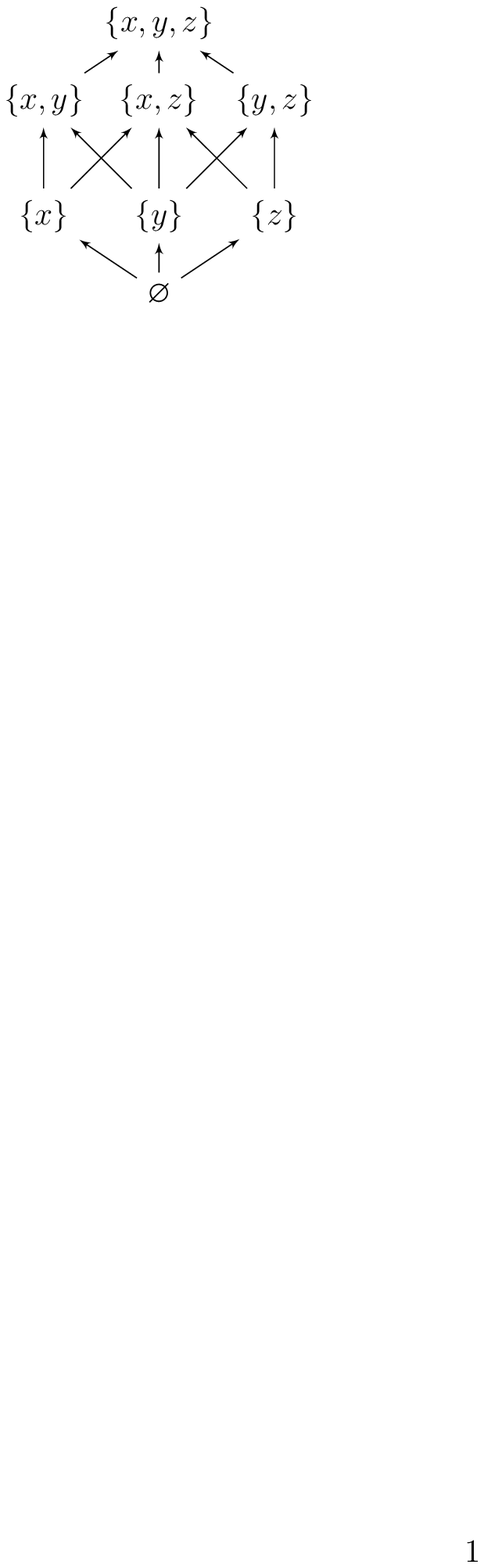}
			\label{subfig:subsetinc}}
	\end{subfigure}
	\begin{subfigure}[Cyclic order]
		{\includegraphics[width=1.3in, bb=100 550 230 670,clip]{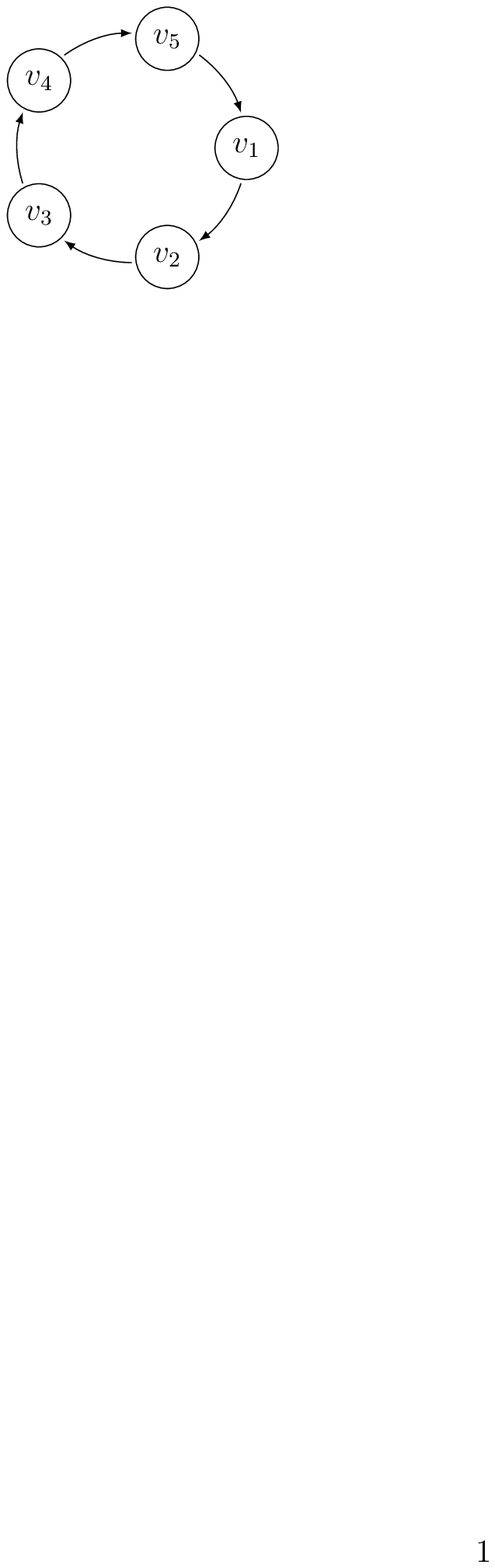}
			\label{subfig:cycle}}
	\end{subfigure}
	\caption{This figure shows the Hasse diagram of ordered sets with different types of order. Fig.~\ref{subfig:integers} represents a total order on $\mathbb{Z}$. Fig.~\ref{subfig:subsetinc} is a partially ordered set formed by the power set of $\{x,y,z\}$ under set inclusion $\subseteq$. Fig.~\ref{subfig:cycle} is the cyclic order on the set $\{v_1,v_2,v_3,v_4,v_5\}$.}
\end{figure}
\begin{example}\label{ex:example1}
	Consider the set of integers $\mathbb{Z}=\{\cdots,-2,-1,0,1,2,\cdots\}$ and the relation $\leq$(less than or equal to), then $(\mathbb{Z},\leq)$ is a totally ordered set. We can write $\cdots \leq -1 \leq0 \leq1<\cdots$ and for all $a,b \in \mathbb{Z}$ it is clear that either $a \leq b$ or $b \leq a$.

	Let us now consider a set $X=\{a,b,c\}$ and its power set $2^X=\{\emptyset,\{a\},\{b\},\{c\},\{a,\linebreak b\},\{a,c\},\{b,c\},\{a,b,c\}\}$. The pair $(2^X,\subseteq)$ is a poset as neither $\{a\} \not\subseteq \{b\}$, nor $\{b\} \not\subseteq \{a\}$. This is true for any of the singletons $\{a\},\{b\},\{c\}$ taken pairwise.    \qquad \eot
\end{example}

Let $(X,\leq)$ be a set with an order, consider $X$ as the vertex set and for each $a \leq b$ there is a directed edge $a \rightarrow b$. This graph represented as $\mathop{\mathcal{O}}\limits_\leq(X)$. The transiticve reduction of this graph is called the \emph{Hasse diagram} and is represented as $\mathop{\mathcal{H}s}\limits_\leq(X)$. A \emph{transitive reduction} of a directed graph $D_1$ is another directed graph $D_2$ with the same vertex set but a reduced edge set, such that if there is a path between vertices $v_1$ and $v_2$ in $D_1$, then there is also such a path in $D_2$.
\begin{example}\label{ex:example2}
	Let us consider the Hasse diagrams for the ordered sets considered in example \ref{ex:example1}. It can be seen that the operation $\leq$ on $\mathbb{Z}$ is a total order the Hasse diagram,Fig.~\ref{subfig:integers}, is a path graph(similar to a line in which each node except the terminating node is connected to two other nodes). That is why a total order is also called a linear order. 
	
	Now let us look at the Hasse diagram for the $(2^X,\subseteq)$, where $X=\{x,y,z\}$ shown in Fg.~\ref{subfig:subsetinc}. Comparing this diagram with the diagram for total order in Fig.~\ref{subfig:integers}, the difference between partial and total order become evident. In total order we can compare any two elements as there is a path between them. Whereas for the partial order as shown in Fig.~\ref{subfig:subsetinc} there is no path between $\{x\}$ and $\{y\}$. This is also true for all the singletons pairwise and all the subsets of size $2$ such as $\{x,y\}$. 
	\qquad \eot
\end{example}

\emph{Ternary relaiton} defined for three sets $X,Y,Z$, is the subset of $X \times Y \times Z$ and is an extension of the binary relation. Using this notion,
\begin{definition}\label{def:corder}
	 Define a ternary relation $\gamma$, that yields triples $[a,b,c]$ such that if one proceeds from $a \rightarrow c$, one has to pass through $b$. For $a,b,c,d \in X$, if this ternary relation $\gamma$ satisfies:
	\begin{compactenum}[1$^o$]
		\item \textbf{Cyclicity:} if $[a,b,c]$, then $[b,c,a]$
		\item \textbf{Antisymmetry:} if $[a,b,c]$ then not $[c,b,a]$
		\item \textbf{Transitivity:} if $[a,b,c]$ and $[a,c,d]$, then $[a,b,d]$
		\item \textbf{Totality:} if $a,b,c$ are distinct then either,$[a,b,c]$ or $[c,b,a]$
	\end{compactenum}      
	then it is called a \emph{cyclic order}.
\end{definition}
  If the \emph{totality condition} is not satisfied we get a \emph{partial cyclic order}, but in this paper we will restrict ourselves to the total cyclic order.
\begin{example}\label{ex:example3}
	Let us explain the triples $[a,b,c]$ yielded by the ternary relation $\gamma$. It means that the points are ordered so that in order to move from $a$ to $c$ one has to pass through $b$. Thus, the order is defined using three elements of the set rather than two as in the total/partial orders. It can be seen that if we have a cycle represented as $v_1\rightarrow v_2\rightarrow v_3\rightarrow v_4\rightarrow v_5\rightarrow v_1$, then any subset of three elements in this sequence(as long as we move in the same direction) satisfies the properties of cyclic order, defined above. To illustrate the choices we have $[v_4,v_5,v_1]$ and $[v_5,v_2,v_4]$ are valid choices but $[v_4,v_2,v_5]$ and $[v_1,v_5,v_4]$ are not. It can be seen that the Hasse diagram is of this particular order is a cyclic graph, and hence the name cyclic order. In the Hasse diagram we ignore all the ordered sets like $[v_2,v_4,v_1]$ as it is a combination of $[v_2,v_3,v_4]$ and $[v_4,v_5,v_1]$.
\end{example}
\begin{figure}
	\centering
	\begin{subfigure}[Proximity space]
		{\includegraphics[width=1.3in, bb=80 450 270 560,clip]{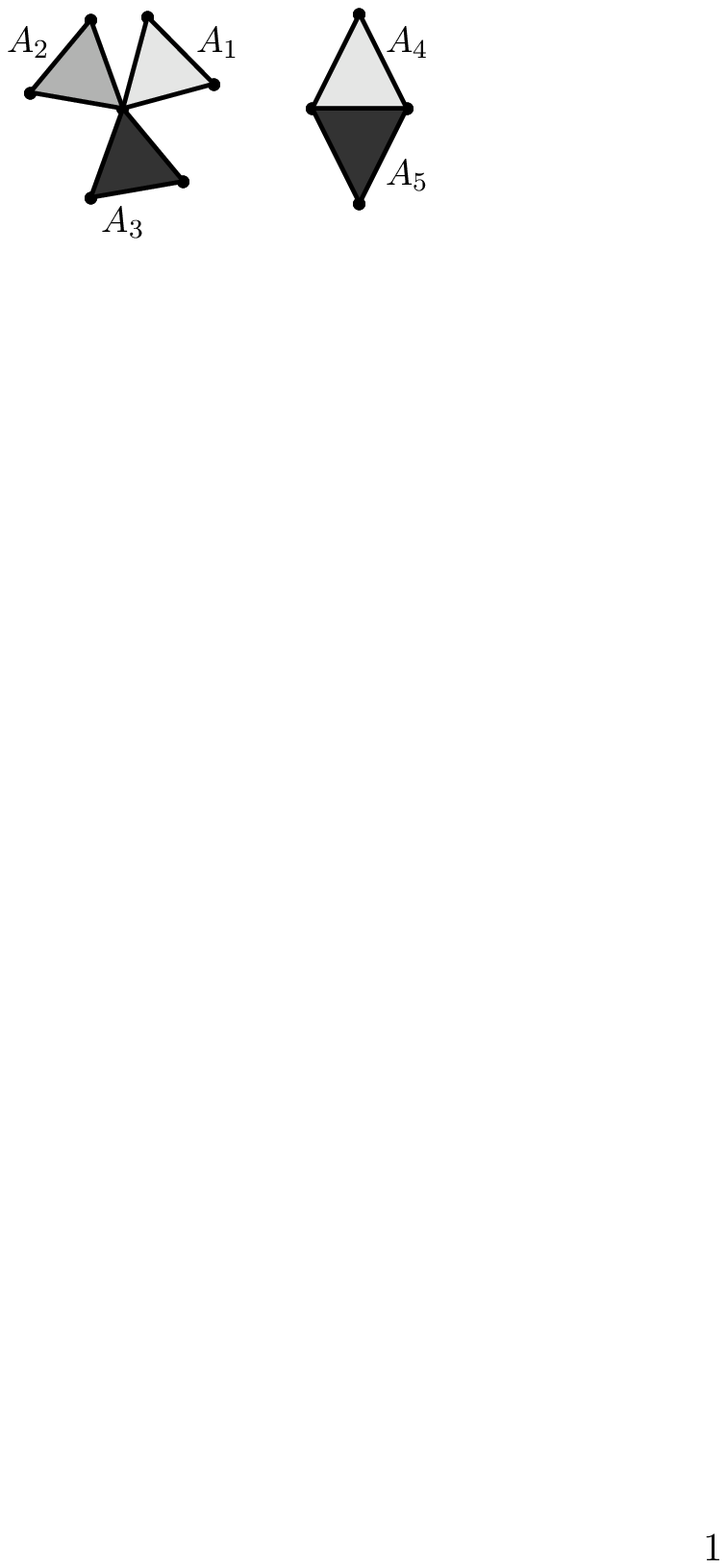}
			\label{subfig:proximity}}
	\end{subfigure}
	\begin{subfigure}[Spatial proximity graph]
		{\includegraphics[width=1.3in, bb=100 580 200 700,clip]{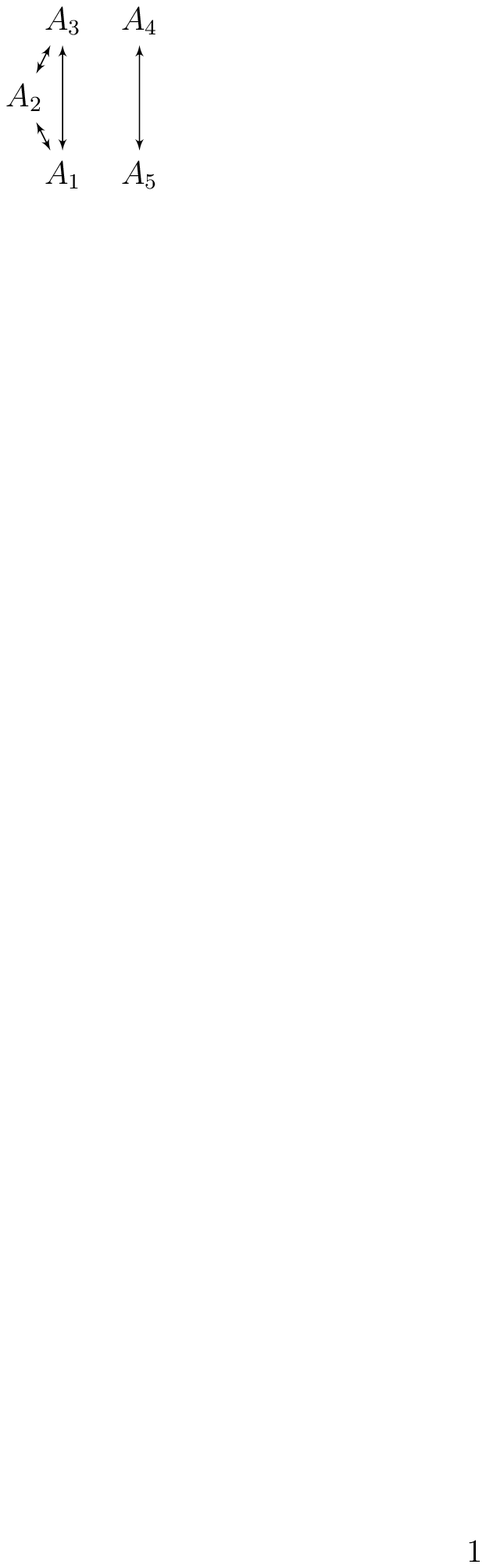}
			\label{subfig:Lodato}}
	\end{subfigure}
	\begin{subfigure}[Descriptive proximity graph]
		{\includegraphics[width=1.3in, bb=100 580 200 700,clip]{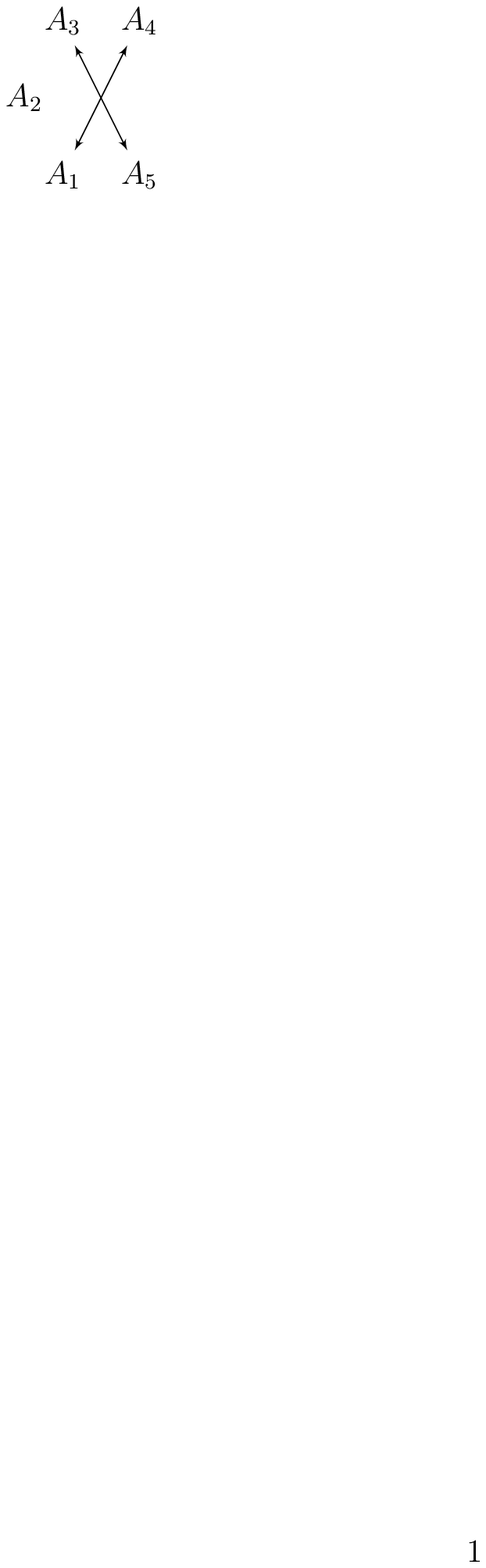}
			\label{subfig:desLodato}}
	\end{subfigure}
	\caption{This figure illustrates the concept of proximity spaces. Fig.~\ref{subfig:proximity} is the space $X$ that has been equipped by proximity relations. Fig.~\ref{subfig:Lodato} is the graph obtained by considering the spatial Lodato proximity $\near$. Fig.~\ref{subfig:desLodato} is the graph obtained using the descriptive Lodato proximity $\dnear$.} 
\end{figure}

Proximity is used to study whether sets in a space are near or far. It can be defined as a binary realtion on a space that yields pairs of subsets that are near each other. We will use the Smirnov functional notation for proximity i.e. $\near(A,B)=0$ if $A$ and $B$ are near and $\near(A,B)=1$ if they are far. 

Different criteria for nearness yield varying proximities. Suppose $A,B \subset X$, $A \cap B \Rightarrow \near(A,B)=0$ is the centeral axiom of C\u{e}ch\cite{Cech1966} and Lodato spatial\cite{Lodato1964},\cite{Lodato1966} proximities. Wallman proximity\cite{Wallman1938} uses $\near(A,B)=0 \Leftrightarrow cl\,A \cap cl\, B \neq \emptyset$, where $cl\, A$ is the closure of $A$. Strong proximity\cite{peters2015strongly} uses, $int\, A \cap int\, B \neq \emptyset \Rightarrow \sn(A,B)=0$, where $int\, A$ is the interior of $A$. The idea of a proximity can be extended to near sets using the notion of a probe function, $\phi:2^X \rightarrow \mathbb{R}^n$, that assigns a feature vector to subsets of the space. We generalize the classical intersection as:
\begin{align*}
A \dcap B= \{x\in A \cup B: \phi(x) \in \phi(A)\,or\,\phi(x) \in \phi(B)\}
\end{align*} 
The descriptive Lodato proximity\cite{Peters2016CP} uses $A \dcap B \neq \emptyset \Rightarrow \dnear(A,B)=0$ and the descriptive strong proximity uses $int\,A \dcap int\,B \neq \emptyset \Rightarrow \snd(A,B)=0$.
We can also get a proximity graph by treating the subsets of a space as the vertices and drawing an edge between $A$ and $B$ if $\near(A,B)=0$.
\begin{example}\label{ex:example4}
	Let us consider the space $X$ with the family of subsets $\{A_1,A_2,A_3,A_4,\linebreak A_5\}$ as shown in Fig.~\ref{subfig:proximity}. It can be seen that if we consider the spatial Lodato proximity then $\{A_1,A_2,A_3\}$ are pairwise proximal as they share a common point. $\near(A_4,A_5)=0$ as they share a common edge. Now, before we move on to drawing the proximity graph for spatial Lodato $\near$, it must be noted that all the proximity relations that we have considered are symmetric i.e. $\near(A,B)=0 \Leftrightarrow \near(B,A)=0$. This is different from the order relations that are antisymmetric. Hence, the proximity diagram for spatial Lodato shown in Fig.~\ref{subfig:Lodato} is an undirected graph. Now let us consider the descriptive Lodato proximity $\dnear$. We can see that $\dnear(A_1,A_4)=0$ and $\dnear(A_3,A_5)=0$ as their interior has the same color. The proximity diagram in this case, as shown in Fig.~\ref{subfig:desLodato}, is also undirected as it is a symmetric relation. 
	\qquad \eot
\end{example}
\section{Order Induced Proximities}\label{sec:results}
Previous works define proximity using a set of axioms. In this work we use different order relations to induce a proximity relation. We have introduced the notion of an ordered set which is a pair $(X,\leq)$, where $X$ is a set and $\leq$ is an order relation. We generalize this structure using the notion of a relator space $(X,\mathcal{R})$, $X$ is a set and $\mathcal{R}$ is a family of relations.
We begin by studying the proximity induced by a partial order defined in def.~\ref{def:porder}.
\begin{definition}\label{def:poprox}
	Let there be a relator space $(X,\{\leq\})$, where $\leq$ is a partial order then we define for $a,b \in X$
	\begin{align*}
	 	(a \leq b) \wedge(\not\exists x \in X \setminus \{a,b\}\; s.t.\;a\leq x \leq b) \Leftrightarrow\; \ponear{\gamma}(a,b)=0 \\
	 	(a \leq b) \wedge (\exists x \in X \setminus \{a,b\}\; s.t.\;a\leq x \leq b) \Leftrightarrow\; \ponear{\gamma}(a,b)=1
	\end{align*}
	Here, $\ponear{\leq}$ is the proximity induced by the partial order $\leq$.
\end{definition}
We formulate the following theorem regarding the properties of the proximity induced by the partial order $\ponear{\leq}$.
\begin{theorem}\label{thm:ponearprop}
	Let $(X,\{\leq\})$ be a relator space where $X$ is a set and $\leq$ is a partial order. Let $a,b,c \in X$ be three elements in $X$. Then, the proximity induced on $X$ by the partial order $\leq$ represented by $\ponear{\leq}$ satisfies the following conditions:
	\begin{compactenum}[1$^o$]
		\item \textbf{Reflexivity:} $\ponear{\leq}(a,a)=0$
		\item \textbf{Antisymmetry:} if $\ponear{\leq}(a,b)=0$ and $\ponear{\leq}(b,a)=0$ then $a=b$
		\item \textbf{Antitransitivity:} if $\ponear{\leq}(a,b)=0$ and $\ponear{\leq}(b,c)=0$ then $\ponear{\leq}(a,c)=1$
	\end{compactenum} 
\end{theorem}
\begin{proof}
	\begin{compactenum}[$1^o$]
		\item It follows directly from the reflexivity condition of the partial order as defined in def.~\ref{def:porder} stating that $a \leq a$ and the antisymmetry condition which states that if $a \leq b$ and $b \leq a$, then $a=b$. Hence, there is no distinct $x \in X$ such that $a \leq x \leq a$, the only such $x$ is $a$. Thus, from def.~\ref{def:poprox} we get $\ponear{\leq}(a,a)=0$.
		\item From def.~\ref{def:poprox} it can be seen that $\ponear{\leq}(a,b)=0$ means that $a \leq b$ and $\ponear{\leq}(b,a)=0$ means that $b \leq a$. Now invoking the antisymmetry property of the partial order as defined in def.~\ref{def:porder}. Thus, if $a \leq b$ and $b \leq a$ then $a=b$.
		\item From def.~\ref{def:poprox}, $\ponear{\leq}(a,b)=0$ states $a \leq b$ and $\ponear{b,c}=0$ is equivalaent to $b \leq c$. This gives us $a \leq b \leq c$. From def.~\ref{def:poprox} we an see that $\ponear{\leq}(a,c)=1$ as there exists an $x\ in X$ such that $a \leq x=b \leq c$.		
	\end{compactenum}
\end{proof}
Given a realtor space $(X,\{\leq\})$ we draw a graph over the vertex set $X$ such that there exists an edge between $a \rightarrow b$ if $a \leq b$. This graph is represented as $\mathop{\check{\mathcal{O}}}\limits_\leq(X)$. We can draw another graph over the vertex set $X$ such that an edge $a \rightarrow b$ exists if $\ponear{\leq}(a,b)=0$. This graph is represented as $\mathop{\check{\mathcal{P}}}\limits_\leq(X)$.
\begin{lemma}\label{lm:reduction}
	Let $(X,\{\leq\})$ be a relator space where $X$ is a set and $\leq$ be a partial order. Then if there exists a path between two vertices in the graph $\mathop{\check{\mathcal{O}}}\limits_\leq(X)$, then there also exists a path between these vertices in $\mathop{\check{\mathcal{P}}}\limits_\leq(X)$.
\end{lemma}
\begin{proof}
	Suppose there is a path between two vertices $a$ and $b$ in $\mathop{\check{\mathcal{O}}}\limits_\leq(X)$. This means that for $a,b \in X \;a\leq b$. Transitivity of the partial order(def.~\ref{def:porder}) dictates that there can be two cases. Either there exists no $x \in X$ such that $a \leq x \leq b$ or there exists a family of elements $\{d_i\} \in X$ such that $a \leq d_1 \leq  \cdots \leq  d_n \leq b$ and no element $x \in X$ can be inserted at any location in this sequence. For the case $\not \exists x \in X\;s.t. \; a \leq x\leq b$, from def.~\ref{def:poprox} we have $\ponear{\leq}(a,b)=0$. Hence, there exists an edge $a \rightarrow b$ in the graph $\mathop{\check{\mathcal{P}}}\limits_\leq(X)$. Now for the case in which $a \leq d_1 \leq \cdots \leq d_n \leq b$ such that no element can be inserted in this sequnce. From def.~\ref{def:poprox} we can write $\ponear{\leq}(a,d_1)=\ponear{\leq}(d_1,d_2)=\cdots=\ponear{\leq}(d_{n-1},d_n)=\ponear{\leq}(d_n,b)=0$. Hence, there exists a sequence of edges $a \rightarrow d_1 \rightarrow d_2 \rightarrow \cdots \rightarrow d_n \rightarrow b$ in the graph $\mathop{\check{\mathcal{P}}}\limits_\leq(X)$ constituting a path between $a$ and $b$.
\end{proof}
\begin{theorem}\label{thm:proxhassepo}
	Let $(X,\{\leq\})$ be a realtor space where $X$ is a set, $\leq$ is a partial order and $\ponear{\leq}$ is the induced proximity as defined by def.~\ref{def:poprox}. Then, $\mathop{\check{\mathcal{P}}}\limits_\leq(X)$ is equivalent to $\mathop{\mathcal{H}s}\limits_\leq(X)$. 
\end{theorem}
\begin{proof}
	We know that $\mathop{\mathcal{H}s}\limits_\leq(X)$ is a transitive reduction of $\mathop{\check{\mathcal{O}}}\limits_\leq(X)$. This means that for every path between two vertices in $\mathop{\check{\mathcal{O}}}\limits_\leq(X)$, there exists a path in $\mathop{\mathcal{H}s}\limits_\leq(X)$. From lemma~\ref{lm:reduction} we can see that $\mathop{\check{\mathcal{P}}}\limits_\leq(X)$ has a path between every pair of vertices that are connected by a path in $\mathop{\check{\mathcal{O}}}\limits_\leq(X)$.
\end{proof}
\begin{example}\label{ex:example5}
	Consider $X=\{a,b,c\}$ and its power set $2^X=\{\{a\},\{b\},\{c\},\{a,b\},\{a,c\} \linebreak ,\{b,c\},\{a,b,c\}\}$, as in example~\ref{ex:example1}. It has been discussed that $(2^X,\subseteq)$ is a partially ordered set and can be visually represented as the Hasse diagram shown in Fig.~\ref{subfig:subsetinc}. Let us see how the proximity relations $\ponear{\subseteq}$ are induced by def.~\ref{def:poprox}. It can be seen that  as there is no $A \in 2^X$ such that $\emptyset \subseteq A 
	\subseteq \{x\}$, we have $\ponear{\subseteq}(\emptyset, \{x\})=0$. Whereas $\emptyset \subseteq \{x\} \subseteq \{x,y\}$, hence $\ponear{\subseteq}(\emptyset, \{x,y\})=1$. Moreover, it can be seen that $\{x\} \not\subseteq \{y\}$ and $\{y\} \not\subseteq \{x\}$ thus the two elements $\{x\},\{y\}$ have no order relation between them. This is why $\subseteq$ is called a partial order. Folowing the same argument there is no induced proximity relation $\ponear{\subseteq}$ between $\{x\}$ and $\{y\}$ hence we cannot say whether they are far or near based on proximity induced by partial order $\subseteq$. It can be seen from Thm.~\ref{thm:proxhassepo} that the proximity graph $\mathop{\check{\mathcal{P}}}\limits_\leq(2^X)$ is the same as the Hasse diagram shown in Fig.~\ref{subfig:subsetinc}. 
	\qquad \eot 
\end{example}
Let us study the proximity induced by the total order defined in def.~\ref{def:torder}.
\begin{definition}\label{def:toprox}
	Let there be a relator space $(X,\{\leq\})$, where $\leq$ is a total order then we define for  $a,b \in X$
	\begin{align*}
		(a \leq b) \wedge (\not\exists x \in X \setminus \{a,b\}\; s.t.\;a\leq x \leq b) \Leftrightarrow\; \tonear{\gamma}(a,b)=0 \\
		(a \leq b) \wedge (\exists x \in X \setminus \{a,b\}\; s.t.\;a\leq x \leq b) \Leftrightarrow\; \tonear{\gamma}(a,b)=1
	\end{align*}
Here, $\tonear{\leq}$ is the proximity induced by the total order $\leq$.
\end{definition}
Let us formulate properties of the proximity induced by the total order $\tonear{\leq}$.
\begin{theorem}\label{thm:tonearprop}
		Let $(X,\{\leq\})$ be a relator space where $X$ is a set and $\leq$ is a total order. Let $a,b,c \in X$ be three elements in $X$. Then, the proximity induced on $X$ by the total order $\leq$ represented by $\tonear{\leq}$ satisfies the following conditions:
	\begin{compactenum}[$1^o$]
		\item \textbf{Reflexivity:} $\tonear{\leq}(a,a)=0$
		\item \textbf{Antisymmetry:} if $\tonear{\leq}(a,b)=0$ and $\tonear{\leq}(b,a)=0$ then $a=b$
		\item \textbf{Antitransitivity:} if $\tonear{\leq}(a,b)=0$ and $\tonear{\leq}(b,c)=0$ then $\tonear{\leq}(a,c)=1$
		\item \textbf{Totality:} either $\big( \exists A=\{a,d_1,\cdots,d_n,b\}  \subseteq X \; s.t. \; \mathop{\sum}\limits_{i=1}^{n+1}\tonear{\leq}(A_i,A_{i+1})=0 \big)$ or $\big( \exists B=\{b,e_1,\cdots,e_m,a\}  \subseteq X \; s.t. \; \mathop{\sum}\limits_{i=1}^{m+1}\tonear{\leq}(B_i,B_{i+1})=0 \big)$
	\end{compactenum} 
where $A_i$ is the $i-$th element of set $A$ and $m,n \in \mathbb{Z}^+$.
\end{theorem}
\begin{proof}
	\begin{compactenum}[1$^o$]
		\item From reflexivity property of total order(def.~\ref{def:torder}) we get $a \leq a$. If we take into account the antisymmetry property of the total order(def.~\ref{def:torder}), which states that  if $a \leq b$ and $b \leq a$ then $a=b$. We conclude that $a \leq b \leq a$ dictates that $b=a$. From def.~\ref{def:toprox} we can say that as no distinct $x \in X$ such that $a \leq x \leq a$ then $\tonear{\leq}(a,a)=0$.
		\item From def.~\ref{def:toprox} we an establish that $\tonear{\leq}(a,b)=0$ means $a \leq b$ and $\tonear{\leq}(b,a)=0$ means $b \leq a$. Now from the antisymmetry property of the total order as defined in def.~\ref{def:torder}, w can conclude that if $a \leq b$ and $b \leq a$ then $a=b$.
		\item From def.~\ref{def:toprox} we can say that $\tonear{\leq}(a,b)=0$ means $a \leq b$ and $\tonear{\leq}(b,c)=0$. Thus $a \leq b \leq c$. Now from def.~\ref{def:toprox} we know that if there  exists $x\in X$ such that $a \leq x \leq b$ then $\tonear{\leq}(a,b)=1$. Hence, $\tonear{\leq}(a,c)=1$.
		\item From the totality property of the total order(def.~\ref{def:torder}) we can conclude that either $a \leq b$ or $b \leq a$. Using the transitive property of the total order(def.~\ref{def:torder}) we conclude that $a \leq b$ means that either there is no $x \in X$ such that $a \leq x \leq b$ in which case from def.~\ref{def:toprox} $\tonear{\leq}(a,b)=0$, or there exists a set of elements $\{d_1,\cdots,d_n\}\subseteq X$ such that $a \leq d_1 \leq \cdots \leq d_n \leq b$. Assuming that there exists no $x \in X$ that can be added to this chain at any point. In this case using the def.~\ref{def:toprox} we can conclude that $\tonear{\leq}(a,d_1)=\tonear{\leq}(d_1,d_2)=\cdots=\tonear{\leq}(d_{n-1},d_n)=\tonear{\leq}(d_n,b)=0$. This can be written as $\exists A=\{a,d_1,\cdots,d_n,b\}\subseteq X$ such that $\mathop{\sum}\limits_{i=1}^{n+1}\tonear{\leq}(A_i,A_{i+1})=0 $. Replicating this argument for $b \leq a$ we can obtain $\exists B=\{b,e_1,\cdots,e_m,a\} \subseteq X$ such that $\mathop{\sum}\limits_{i=1}^{m+1}\tonear{\leq}(B_i,B_{i+1})=0$.
	\end{compactenum}
\end{proof}
For the relator space $(X,\{\leq\})$ where $X$ is a set and $\leq$ is a total order then we have a graph with $X$ as the vertex set and an edge $a \rightarrow b$ for $a,b \in X$ such that $a \leq b$. This is represented as $\mathop{\bar{\mathcal{O}}}\limits_\leq(X)$. Another graph over the vertex set $X$ is obtained by adding an edge $a \rightarrow b$ if $\tonear{\leq}(a,b)=0$, and is represented as $\mathop{\bar{\mathcal{P}}}\limits_\leq(X)$. 
\begin{lemma}\label{lm:toreduction}
	Let $(X,\{\leq \})$ be arelator space where $X$ is a space and $\leq$ a total order. Then, if there exists a path between two vertices in the graph $\mathop{\bar{\mathcal{O}}}\limits_\leq(X)$, then there also exists a path between these vertices in $\mathop{\bar{\mathcal{P}}}\limits_\leq(X)$.
\end{lemma}
\begin{proof}
	A path $a \rightarrow b$ in $\mathop{\bar{\mathcal{O}}}\limits_\leq(X)$ means $a \leq b$. By the transitivity of total order(def.~\ref{def:torder}) we can conclude that this leads to one of the two cases. Either $\not \exists x\ in X\; s.t. \; a \leq x \leq b $ or $\exists \{d_i\} \in X \; s.t. a\leq d_1 
	\leq \cdots \leq d_n \leq b$ such that no $x \in X$ can be inserted at any location in this sequence. For the first case from def.~\ref{def:toprox} we can write $\tonear{\leq}(a,b)=0$ and hence there is an edge $a \rightarrow b$ in $\mathop{\bar{\mathcal{P}}}\limits_\leq(X)$. For the second case usin def.~\ref{def:toprox} we can write $\tonear{\leq}(a,d_1)=\tonear{\leq}(a,d_1)=\cdots=\tonear{\leq}(d_{n-1},d_n)=\tonear{\leq}(d_n,b)=0$. Thus, there exist a sequence of edges $a \rightarrow d_1 \cdots \rightarrow d_{n}\rightarrow b$ in $\mathop{\bar{\mathcal{P}}}\limits_\leq(X)$ constituting a path between $a$ and $b$.
\end{proof}
\begin{theorem}\label{thm:proxhasseto}
	Let $(X,\{\leq\})$ be a realtor space where $X$ is a set, $\leq$ is a partial order and $\ponear{\leq}$ is the induced proximity as defined by def.~\ref{def:toprox}. Then, $\mathop{\check{\mathcal{P}}}\limits_\leq(X)$ is equivalent to $\mathop{\mathcal{H}s}\limits_\leq(X)$. 
\end{theorem}
\begin{proof}
	From lemma~\ref{lm:toreduction} we know that there is a path between two vertices in $\mathop{\check{\mathcal{P}}}\limits_\leq(X)$ if there is a path between them in $\mathop{\bar{\mathcal{O}}}\limits_\leq(X)$. Hence, $\mathop{\check{\mathcal{P}}}\limits_\leq(X)$ is a transitive reduction of $\mathop{\bar{\mathcal{O}}}\limits_\leq(X)$. We know that by definition $\mathop{\mathcal{H}s}\limits_\leq(X)$ is also a transitive reduction of $\mathop{\bar{\mathcal{O}}}\limits_\leq(X)$. 
\end{proof}
\begin{example}\label{ex:example6}
	Now consider $\mathbb{Z}$ as in  example \ref{ex:example1}, with the total order $\leq$. $(\mathbb{Z},\{\leq \})$ is a totally ordered set that can be visualized by the aid of the Hasse diagram displayed in Fig.~\ref{subfig:integers}. 
	
	Now, we move on to understanding how this order induces a proximity $\tonear{\leq}$ as per def.~\ref{def:toprox}. It can be seen that as there is no $x \in \mathbb{Z}$ such that $0 \leq x \leq 1$, we have $\tonear{\leq}(0,1)=0$. Moreover, as $0 \leq 1 \leq 2$, we have $\tonear{\leq}(0,2)=1$. It can also bee seen that $1 \not\leq 0$, hence we cannot talk about $\tonear{\leq}(1,0)$. Thus, the proximity inherits antisymetric nature of the underlying order as it is symmetric if both elements are the same. 
	
	We can see that if $a,b \in \mathbb{Z}$ then either $a \leq b$ or $b \leq a$, which is the totality(connex) property. Thus, results in the linear structure of fig.~\ref{subfig:integers} and is the reason why thew total order is also called a linear order. The connex property dictates that any two elements $a,b \in \mathbb{Z}$ are related in one of two ways. Either there is a set of elements $\{a,x_1,\cdots,x_n,b\}$ in $\mathbb{Z}$ such that $\tonear{\leq}(a,x_1)=0 \wedge\tonear{\leq}(x_1,x_2)=0\wedge \cdots \wedge \tonear{\leq}(x_n,b)=0$, or there exists a set $\{b,y_1,\cdots,y_m,a\}$ such that $\tonear{\leq}(b,y_1)=0\wedge \tonear{\leq}(y_1,y_2)=0\wedge \cdots \wedge \tonear{\leq}(y_n,a)=0$. Moreover, by Thm.~\ref{thm:proxhasseto} the proximity graph $\mathop{\check{\mathcal{P}}}\limits_\leq(X)$  is the same as the Hasse diagram shown in Fig.~\ref{subfig:integers}. \qquad \eot
\end{example}
We move on to the study of proximity induced by cyclic order defined in def.~\ref{def:corder}.
\begin{definition}\label{def:cprox}
	Let there be a realtor space $(X,\{\gamma\})$, where $\gamma$ is a cyclic order which yields triples $[a,b,c]$ implying that moving from $a$ to $c$ one passes through $b$. For  $a,b \in X$ we define
	\begin{align*}
		([a,b,c]) \wedge (\not\exists x \in X \setminus \{a,b\}\; s.t.\;[a,x,b]) \Leftrightarrow\; \cnear{\gamma}(a,b)=0 \\
		([a,b,c]) \wedge (\exists x \in X \setminus \{a,b\}\; s.t.\;[a,x,b]) \Leftrightarrow\; \cnear{\gamma}(a,b)=1
	\end{align*}
	where $\cnear{\gamma}$ is the proximity induced by the cyclic order $\gamma$.
\end{definition}
Now, we establish the properties of proximity induced by the cyclic order $\cnear{\gamma}$.
\begin{theorem}\label{thm:cnearprop}
	Let $(X,\{\gamma\})$ be a relator space where $X$ is a set and $\gamma$ is a cyclic order. Let $a,b,c \in X$ be three elements in $X$. Then, the proximity induced on $X$ by the total order $\gamma$ represented by $\cnear{\gamma}$ satisfies the following conditions:
	\begin{compactenum}[1$^o$]
		\item \textbf{Irreflexivity:} $\cnear{\gamma}(a,a)=1$
		\item \textbf{Antisymmetry:} if $\cnear{\gamma}(a,b)=0$ then $\cnear{\gamma}(b,a)=1$
		\item \textbf{Antitransitivity:} if $\cnear{\gamma}(a,b)=0$ and $\cnear{\gamma}(b,c)=0$ then $\cnear{\gamma}(a,c)=1$
		\item \textbf{Totality:} either $\big( \exists A=\{a,d_1,\cdots,d_n,b\}  \subseteq X \; s.t. \; \mathop{\sum}\limits_{i=1}^{n+1}\cnear{\gamma}(A_i,A_{i+1})=0 \big)$ or $\big( \exists B=\{b,e_1,\cdots,e_m,a\}  \subseteq X \; s.t. \; \mathop{\sum}\limits_{i=1}^{m+1}\cnear{\gamma}(B_i,B_{i+1})=0 \big)$
		\item \textbf{Cyclicity:} if $ \big( \exists A=\{a,d_1,\cdots,d_m,b,e_1,\cdots,e_n,c\} \; s.t. \; \mathop{\sum}\limits_{i=1}^{n+m+2}\cnear{\gamma}(A_i,A_{i+1})=0  \big)$ then $\big( \exists B=\{b,f_1,\cdots,f_o,c,g_1,\cdots,g_p,a\} \; s.t. \; \mathop{\sum}\limits_{i=1}^{o+p+2}\cnear{\gamma}(B_i,B_{i+1})=0  \big)$
	\end{compactenum} 
where $A_i$ is the $i-$th element of set $A$ and $m,n,o,p \in \mathbb{Z}^+$.
\end{theorem}
\begin{proof}
	\begin{compactenum}[1$^o$]
		\item Cyclic order requires three elements to define it. If there exists $[a,b,c]$ meaning that the elements are ordered in such a way that moving from $a$ to $c$ we must pass from $b$, then from the cyclicity property of the cyclic order(def.~\ref{def:corder}) we have $[b,c,a]$. Now using the transitivity property of the cyclic order(def.~\ref{def:corder}) if $[a,b,c]$ and $[b,c,a]$ then we can write $[a,c,a]$. Hence, from def.~\ref{def:cprox} we can coclude that $\cnear{\gamma}(a,a)=1$ as there is $c \in X$ such that $[a,c,a]$.
		\item It must be noted that symmetry in a function corresponds to switching the locations of input variables.From def.~\ref{def:cprox} we can conclude that $\cnear{\gamma}(a,b)=0$ means that $[a,b,c]$ and $\cnear{\gamma}(b,a)=0$ means that $[b,a,c]$. Using the cyclicity property of the cyclic order(def.~\ref{def:corder}) we can see that $[b,a,c] \Rightarrow [a,c,b] \Rightarrow [c,b,a]$. Now using the antisymmetry property of the cyclic order(def.~\ref{def:corder}) if $[a,b,c]$ then not $[c,b,a]$.
		\item From def.~\ref{def:cprox} we can conclude that $\cnear{\gamma}(a,b)=0$ means that there exists a path from $a$ passing through $b$ and moving onwards. Moreover, there exists no such $x\in X$, for which $[a,x,b]$. Using the same reasoning we can conclude from $\cnear{\gamma}(b,c)=0$ that there exists no such $x \in X$ for which $[b,x,c]$. Combining the two using transitivity of the cyclic order(def.~\ref{def:corder}) we can write $[a,b,c]$. Now, using the def.~\ref{def:cprox} we can see that $\cnear{\gamma}(a,c)=1$. 
		\item From the totality property of the cyclic order(def.~\ref{def:corder}) it can seen that for three distinct points $a,x,b \in X$ either $[a,x,b]$ or $[b,x,a]$. Using the transitive property of the cyclic order(def.~\ref{def:corder}) we can conclude that there exists a set of elements $\{d_1,\cdots,d_n\} \subseteq X$ such that $[a,d_1,d_2],[d_1,d_2,d_3],\cdots,[d_{n-2},d_{n-1},d_n],[d_{n-1},d_n,b]$. Moreover there exists no arbitrary $x\in X$ such that $[a,x,d_1],[d_1,x,d_2],\cdots,[d_{n-1},\linebreak x,d_n],[d_n, x,b]$.This means that there exists a sequence of adjacent elements traversing which we can go from $a$ to $b$. Using def.~\ref{def:cprox} we can rewrite this as $\cnear{\gamma}(a,d_1)=\cnear{\gamma}(d_1,d_2)=\cdots=\cnear{\gamma}(d_{n-1},d_n)=\cnear{\gamma}(d_n,b)=0$. This can be rewritten as $\exists A={a,d_1,\cdots,d_n,b} \subseteq X$ such that $\mathop{\sum}\limits_{i=1}^{n+1}\cnear{\gamma}(A_i,A_{i+1})=0$. Using similar line of argument we can conclude that $[b,x,a]$ can be rewritten as  $\exists B={b,e_1,\cdots,e_m,a} \subseteq X$ such that $\mathop{\sum}\limits_{i=1}^{m+1}\cnear{\gamma}(B_i,B_{i+1})=0$.
		\item We know that the ternary relation $\gamma$ yields $[a,b,c]$ which means that elements are ordered in such a way that one has to go through $b$ when moving from $a$ to $c$, from def.~\ref{def:corder}. We know that $\exists A=\{a,d_1,\cdots,d_m,b,e_1,\cdots,e_n,c\} \; s.t. \; \mathop{\sum}\limits_{i=1}^{n+m+2}\cnear{\gamma}(A_i,A_{i+1}) \linebreak =0$ means that $\cnear{\gamma}(a,d_1)=\cnear{\gamma}(d_1,d_2)=\cdots=\cnear{\gamma}(d_{m-1},d_m)=\cnear{\gamma}(d_m,b)=\cnear{\gamma}(b,e_1)=\cnear{\gamma}(e_{n-1},e_n)=\cnear{\gamma}(e_n,c)=0$. Let us consider two adjacent terms in this chain e.g. $\cnear{\gamma}(a,d_1)=\cnear{\gamma}(d_1,d_2)=0$. From def.~\ref{def:cprox} $\cnear{\gamma}(a,d_1)=0$ means that $a$ and $d_1$ are adjacent in some path i.e. $[a,d_1,y]$ where $y \in X$ and there exists no $x \in X$ such that $[a,x,d_1]$ . Combining this with $\cnear{\gamma}(d_1,d_2)=0$ which yields that $[d_1,d_2,\tilde{y}]$ for some $\tilde{y} \in X$ and there exists no $\tilde{x} \in X$ such that $[d_1,\tilde{x},d_2]$, we get $[a,d_1,d_2]$. If we go on doing this we can get $[a,b,c]$. Now using the cyclicity property of the cyclic order(def.~\ref{def:corder}) if $[a,b,c]$ then $[b,c,a]$. Using the transativity property of the cyclic order(def.~\ref{def:corder}) we can decompose $[b,c,a]$ into $[b,c,x]$ and $[c,x,a]$ where $x \in X$. Using subsequent decompositions such that we have paths in which all the elements are adjacents(i.e. they cannot be further decomposed in this way) we get $[b,f_1,f_2],[f_1,f_2,f_3],\cdots,[f_{o-1},f_{o},c],[f_{o},c,g_1],[c,g_1,g_2],\cdots,[g_{p-1},g_p,a]$ where $\{f_i\} \linebreak ,\{g_j\} \in X$ and $p,q \in \mathbb{Z}^+$. Using def.~\ref{def:cprox} $[b,f_1,f_2]$ yields $\cnear{\gamma}(b,f_1)=0$. In a similar fashion we can write $\cnear{\gamma}(b,f_1)=\cnear{\gamma}(f_1,f_2)=\cdots=\cnear{\gamma}(f_{o-1},f_o)=\cnear{\gamma}(f_o,c)=\cnear{\gamma}(c,g_1)=\cdots=\cnear{\gamma}(g_{p-1},g_p)=\cnear{\gamma}(g_p,a)=0$. This can inturn be simplified as $\exists B=\{b,f_1,\cdots,f_o,c,g_1, \cdots,g_p,a\} \subseteq X \; s.t. \; \mathop{\sum}\limits_{i=1}^{o+p+2}\cnear{\gamma}(B_i,B_{i+1})=0$
	\end{compactenum}
\end{proof}
Given a relator space $(X,\{\gamma\})$, where $X$ is a set and $\gamma$ is a cyclic order, the graph $\mathop{\mathring{\mathcal{O}}}\limits_\leq(X)$ has the vertex set $X$ with a sequence of edges $a \rightarrow b \rightarrow c$ if $[a,b,c]$. The graph $\mathop{\mathring{\mathcal{P}}}\limits_\leq(X)$ also has the vertex set $X$ but there exists an edge $a \rightarrow b$ if $\cnear{\gamma}(a,b)=0$.   
\begin{lemma}\label{lm:creduction}
Let $(X,\{\gamma \})$ be a realtor space where $X$ is a set and $\gamma$ is a cyclic order. Then, if there exists a path in the graph $\mathop{\mathring{\mathcal{O}}}\limits_\gamma(X)$, then there also exists a path between these vertices in $\mathop{\mathring{\mathcal{P}}}\limits_\gamma(X)$.
\end{lemma}
\begin{proof}
  It must be noted that $\gamma$ is a ternary relation and yields $[a,b,c]$ stating that the elements are ordered in such a way that when moving from $a$ to $c$ we must pass through $b$. Hence, a path $a \rightarrow b \rightarrow c$ in $\mathop{\mathring{\mathcal{O}}}\limits_\leq(X)$ means $[a,b,c]$. Using the transitivity of cyclic order(def.~\ref{def:corder}) we can see that there can two cases. Either, there exist no $x,y \in X$ such that $[a,x,b]$ and $[b,y,c]$. Or, there exist $\{d_i\},\{e_j\} \in X$ such that $[a,d_1,d_2],[d_1,d_2,d_3],\cdots,[d_{n-1},d_n,b],[d_n,b,e_1],\cdots,[e_{m-1},e_m,c]$ and there exist no $x \in X$ which can be inserted in this sequence at any place. For the first case using def.~\ref{def:cprox} we can write $\cnear{\gamma}(a,b)=\cnear{\gamma}(b,c)=0$, hence there exists a sequence of edges $a \rightarrow b \rightarrow c$ in graph $\mathop{\mathring{\mathcal{P}}}\limits_\leq(X)$ and hence a path between $a$ and $b$. For the second case using def.~\ref{def:cprox} we can write $\cnear{\gamma}(a,d_1)=\cnear{\gamma}(d_1,d_2)=\cdots=\cnear{\gamma}(d_n,b)=\cnear{\gamma}(b,e_1)=\cnear{\gamma}(e_1,e_2)=\cdots=\cnear{\gamma}(e_m,c)=0$. Thus we have a sequence of edges $a \rightarrow d_1 \rightarrow \cdots \rightarrow d_n \rightarrow b 
  \rightarrow e_1 \rightarrow \cdots \rightarrow e_m \rightarrow c$ in $\mathop{\mathring{\mathcal{P}}}\limits_\leq(X)$ and hence a path between $a$ and $b$.
    
\end{proof}
\begin{theorem}\label{thm:proxhassec}
Let $(X,\{\gamma \})$ be a realtor space where $X$ is a set, $\gamma$ is a cyclic order and $\cnear{\gamma}$ is the induced proximity as defined by def.~\ref{def:cprox}. Then, $\mathop{\mathring{\mathcal{P}}}\limits_\gamma(X)$ is equivalent to $\mathop{\mathcal{H}s}\limits_\gamma(X)$. 
\end{theorem}
\begin{proof}
 From lemma~\ref{lm:creduction} it can be seen that $\mathop{\mathring{\mathcal{P}}}\limits_\gamma(X)$ has a path between every pair of vertices that are connected by a path in $\mathop{\mathring{\mathcal{O}}}\limits_\gamma(X)$. Thus it is a transitive reduction of $\mathop{\mathring{\mathcal{O}}}\limits_\leq(X)$. We know that by definition $\mathop{\mathcal{H}s}\limits_\gamma(X)$ is also a transitive reduction of $\mathop{\mathring{\mathcal{O}}}\limits_\leq(X)$
\end{proof}
\begin{example}\label{ex:example7}
	Consider $X=\{v_1,v_2,v_3,v_4,v_5\}$ with the relation $\gamma$ as in example \ref{ex:example3}. We know that $\gamma$ yields triples $[a,b,c]$ which mean that when moving from $a$ to $c$ one has to pass through $b$. For more on the relation $\gamma$ consult example \ref{ex:example3}. Let us consider the relator space $(X,\{\gamma\})$ such that $\gamma$ is a cyclic order as per def.\ref{def:torder}. The relator space can be visualized as the Hasse diagram shown in Fig.~\ref{subfig:cycle}. 
	
	Let us look at how the cyclic order $\gamma$ induces the proximity $\cnear{\gamma}$. It can be seen that as there exists no $x \in X$ such that $[a,x,b]$ we have $\cnear{\gamma}(a,b)=0$. Here, we must note that the totality property of def.~\ref{def:corder} dictates that any three elements in $X$ are a part of some path. Hidden in the above is an assumption that there exists a c such that $[a,b,c]$ hence the existence of a path between $a$ and $b$ is guarrenteed. Moreover, from Thm.~\ref{thm:proxhassec} the proximity graph $\mathop{\mathring{\mathcal{P}}}\limits_\leq(X)$ is the same as the Hasse diagram shown in Fig.~\ref{subfig:cycle}. \qquad \eot
\end{example}
\section{Applications}\label{sec:apps}
This section considers the possible applications of order induced proximities. 
\subsection{Maximal centroidal vortices on images}\label{subsec:mcv}
For this purpose we consider a triangulated space. A set of triangles with a nonempty intersection is a \emph{nerve}. The nerve with maximal number of triangles is termed a \emph{maximal nuclear cluster}(MNC), and the common intersection of this is the \emph{nucleus}. Each of the triangles in MNC is termed as the \emph{$1$-spoke}($sk_1$). To keep it simple we assume that there is only one MNC in the triangulation. 

Generalizing these structures we have $k$-spokes($sk_k$), that are triangles sharing an intersection with $s_{k-1}$ but not with $sk_{k-2}$. This is a recursive definition for which the base case is $sk_0$ which is the nucleus. Collection of all the $sk_k$ is the \emph{$k-$spoke complex}($skcx_k$). We can see that the MNC is $skcx_1$, and hence spoke complexes generalize the notion of MNC. Associated with the $skcx_k$ is the $k-$maximal cycle($mcyc_k$), that is the cycle formed by centroids of all the $sk_k$. A collection of $\{mcyc_k\}_{k\in \mathbb{Z}^+}$ is termed a \emph{vortex}. For further detail we refer reader to \cite{ahmad2018maximal}.

The notion of order gives us a systematic way of constructing such cycles. Let $\phi:2^K \rightarrow \mathbb{R}$ be a real valued probe function that attaches description to the subsets of $K$. Thus, for $skcx_k$ we  have a set $X_k=\{\phi(\triangle_i) s.t. \triangle_i \in skcx_k\}$. We can use $\leq$ to establish a total order(def.~\ref{def:torder}) $X_k$ such that $\phi(\triangle_1) \leq \cdots \leq \phi(\triangle_{n-1}) \leq \phi(\triangle_n)$. Let us define a ternary relation $\gamma_\leq$ induced by $\leq$ as $(a \leq b \leq c) \lor (b \leq c \leq a) \lor(c \leq a \leq b)  \Leftrightarrow [a,b,c]$.
\begin{lemma}\label{lm:app}
	Let $(X,\leq)$ be a relator space where $X$ is a set and $\leq$ is a total order. Define a ternary relation $\gamma_\leq$ such that for $a,b,c \in X$, $(a \leq b \leq c) \lor (b \leq c \leq a) \lor(c \leq a \leq b)  \Leftrightarrow [a,b,c]$. Then, $\gamma_\leq$ satisfies the following properties:
	\begin{compactenum}[1$^o$]
		\item \textbf{Cyclicity:} if $[a,b,c]$, then $[b,c,a]$
		\item \textbf{Antisymmetry:} if $[a,b,c]$ then not $[c,b,a]$
		\item \textbf{Transitivity:} if $[a,b,c]$ and $[a,c,d]$, then $[a,b,d]$
		\item \textbf{Totality:} if $a,b,c$ are distinct then either,$[a,b,c]$ or $[c,b,a]$
	\end{compactenum}
\end{lemma}
\begin{proof}
	\begin{compactenum}[1$^o$]
		\item From the definition of $\gamma_\leq$ we can see that $[a,b,c]$ is equivalent to $(a \leq b \leq c) \lor (b \leq c \leq a) \lor(c \leq a \leq b)$. It can also be seen by substitution that $[b,c,a]$ is equivalent to $(b \leq c \leq a) \lor (c \leq a \leq b) \lor (a \leq b \leq b)$. Hence $[a,b,c]$ and $[b,c,a]$ are equivalent.
		\item We can see from the definition of $\gamma_\leq$ that $[a,b,c]$ stands for $(a \leq b \leq c) \lor (b \leq c \leq a) \lor(c \leq a \leq b)$. Moreover, not $[c,a,d]$ would stand for $\neg \big((c \leq b \leq a) \lor (b \leq a \leq c) \lor (a \leq c \leq b)\big)$, that can be written as $\neg(c \leq b \leq a) \wedge \neg(b \leq a \leq c) \wedge \neg(a \leq c \leq b)$. It can be seen for each of the conditions of $[a,b,c], \neg(c \leq b \leq a) \wedge \neg(b \leq a \leq c) \wedge \neg(a \leq c \leq b)$ holds. This is because none of the conditions of $[a,b,c]$ are the same as any of the conditions in $[c,b,a]$.
		\item From the definition of $\gamma_\leq$ we can write $[a,b,c]$ as $(a \leq b \leq c) \lor (b \leq c \leq a) \lor(c \leq a \leq b) $ and $[a,c,d]$ as $(a \leq c \leq c=d) \lor (c \leq d \leq a) \lor(d \leq a \leq c) $. Out of the $9$ possible combinations only $4$ can occur. Combination such as $b \leq c \leq a$ and $a \leq c \leq d$ cannot occur as one forces $a \leq c$ and the other forces $c \leq a$ while $a,b,c$ are distinct. The answers for the possible combinations are:
		\begin{align*}
		(a \leq b \leq c)\, and\, (a \leq c \leq d) \Rightarrow (a \leq b \leq c \leq d) \Rightarrow (a \leq b \leq d)\\
		(a \leq b \leq c)\, and\, (d \leq a \leq c) \Rightarrow (d \leq a \leq b \leq c) \Rightarrow (d \leq a \leq b)\\
		(b \leq c \leq a)\, and\, (c \leq d \leq a) \Rightarrow (b \leq c \leq d \leq a) \Rightarrow (b \leq d \leq a)\\
		(c \leq a \leq b)\, and\, (c \leq d \leq a) \Rightarrow (c \leq d \leq a \leq b) \Rightarrow (d \leq a \leq b)
		\end{align*}
		We can confirm from the definition of $\gamma_\leq$ that $(a \leq b \leq d) \lor (b \leq d \leq a) \lor (d \leq a \leq b)$ is equivalent to $[a,b,d]$.
		\item We know that totality of the total order(def.~\ref{def:torder}) implies that for any two elements either  $a \leq b$ or $b \leq a$. When we extend this to three elements we can compute that there are $6$ different possibilities $(a \leq b \leq c),(b \leq c \leq a),(c \leq b \leq a),(c \leq b \leq a),(b \leq a \leq c),(a \leq c \leq b)$. For visualization we can consider these as the permutations of $3$ numbers. One of these six possibilities has to hold as dictated by totality of the underlying total order $\leq$. Further more we know from the definition of $\gamma_\leq$ that $(a \leq b \leq c) \lor (b \leq c \leq a) \lor(c \leq a \leq b)$ is equivalent to $[a,b,c]$ and $(c \leq b \leq a) \lor (b \leq a \leq c) \lor (a \leq c \leq b)$ is equivalent to $[c,b,a]$. We can see that $[a,b,c]$ and $[c,b,a]$ are mutually exclusive and together cover all the six possibilities. Hence, one of them has to occur.
	\end{compactenum}
\end{proof}
\begin{theorem}
	Let $(X,\leq)$ be a relator space where $X$ is a set and $\leq$ is a total order. Define a ternary relation $\gamma_\leq$ such that for $a,b,c \in X$, $(a \leq b \leq c) \lor (b \leq c \leq a) \lor(c \leq a \leq b)  \Leftrightarrow [a,b,c]$. Then, $\gamma_\leq$ is a cyclic order.
\end{theorem}
\begin{proof}
	The proof follows directly from lemma \ref{lm:app} and the definition of cyclic order given in def.~\ref{def:corder}.
\end{proof}
\begin{figure}
	\centering
	\begin{subfigure}[Order induced by orientation]
		{\includegraphics[width=2in, bb=95 530 170 603,clip]{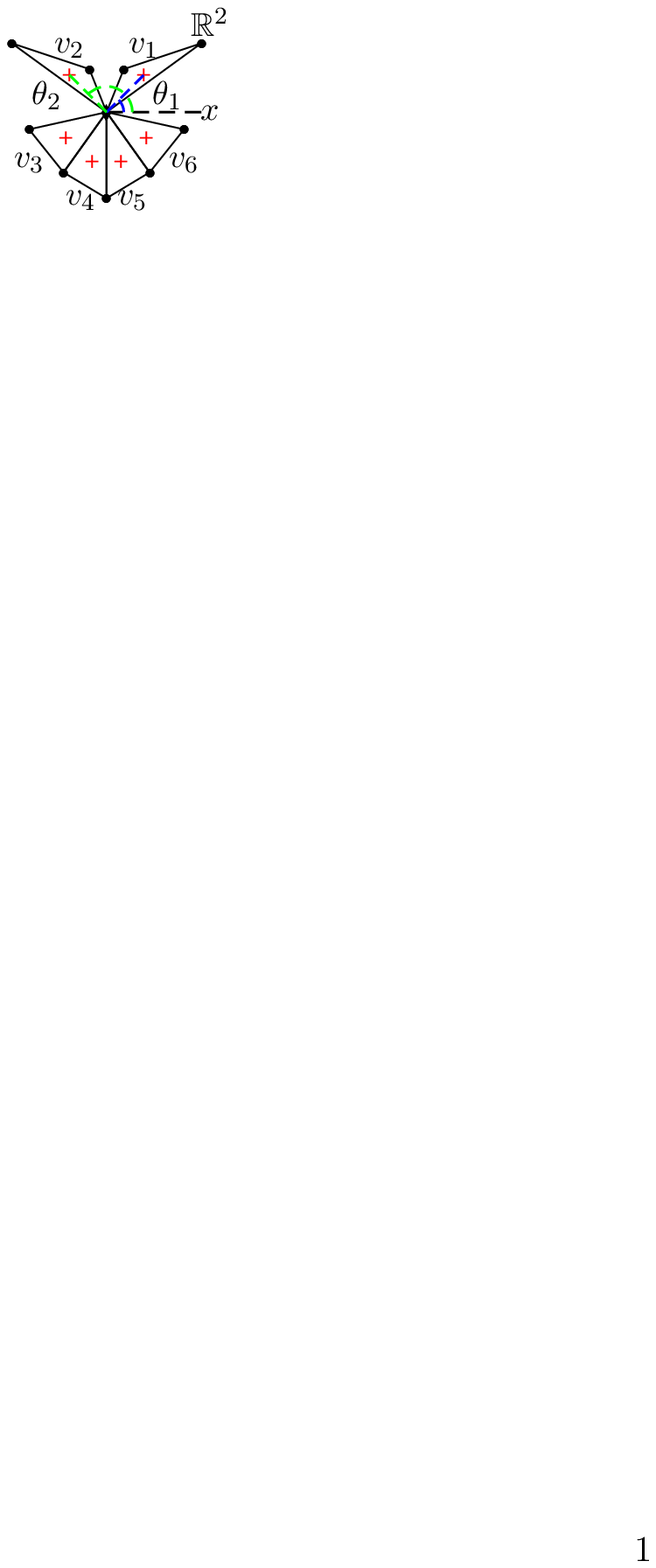}
			\label{subfig:orientorder}}
	\end{subfigure}
	\begin{subfigure}[$1$-maximal cycle]
		{\includegraphics[width=2in, bb=95 530 170 603,clip]{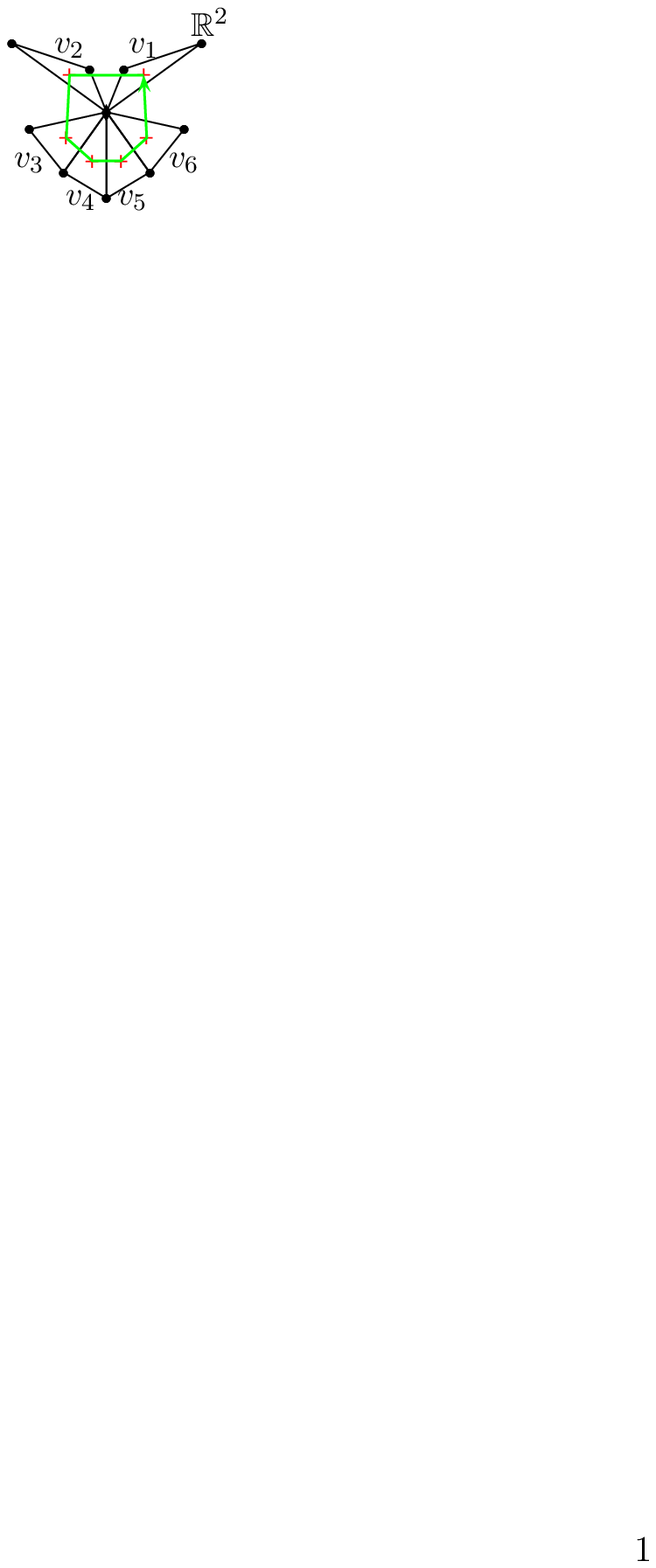}
			\label{subfig:mcyc1}}
	\end{subfigure}
\end{figure}
Now that we have shown that $\gamma_\leq$ is a cyclic order we can induce a proximity on the relator space $(X_k,\gamma_\leq)$, where $X_k$ is the set of real valued descriptions of all triangles in $skcx_k$. Upon this we can induce a proximity $\cnear{\gamma_\leq}$ as defined by def.~\ref{def:cprox}. It can be seen that for $X_k=\{\phi(\triangle_i) \, s.t. \, \triangle_i \in skcx_k\}$ that are ordered as $\phi(\triangle_1) \leq \phi(\triangle_2)\cdots \leq \phi(\triangle_{n-1}) \leq \phi(\triangle_n)$, the induced proximity by the ternary relation $\gamma_\leq$(lemma \ref{lm:app}) would yield $\cnear{\gamma_\leq}(\phi(\triangle_1),\phi(\triangle_2))=\cnear{\gamma_\leq}(\phi(\triangle_2),\phi(\triangle_3))=\cdots=\cnear{\gamma_\leq}(\phi(\triangle_{n-1}),\phi(\triangle_{n}))=\cnear{\gamma_\leq}(\phi(\triangle_n),\phi(\triangle_1))=0$. Now, if we were to draw a path $cnt(\triangle_1)\rightarrow cnt(\triangle_2)\rightarrow \cdots \rightarrow cnt(\triangle_{n-1})\rightarrow cnt(\triangle_n) \rightarrow cnt(\triangle_1)$, where $cnt(\triangle_i)$ is the centroid, it would be a $k-$maximal cycle or $mcyc_k$. It is similar to drawing the proximal graph($\mathop{\mathring{\mathcal{P}}}\limits_\gamma(X_k)$) for $(X,\gamma_\leq)$ but using $cnt(\triangle_i)$ instead of $\phi(\triangle_{n})$. Let us look at an example to clarify this procedure.

\begin{example}
	Consider the triangulation(shown in Figs.~\ref{subfig:orientorder},\ref{subfig:mcyc1}) that is a subset of the $\mathbb{R}^2$. We can see that all the triangles are a part of the MNC or the $skcx_1$. The nucleus is shown as black diamond and the centroids of the triangles are shown as red crosses. It is evident that each of the triangles in $skcx_1$ shares a nonempty intersection with the nucleus and is hence proximal to it and each other as per Lodato proximity. This is an important point, that under the Lodato proximity w.r.t. the nucleus each of the $sk_1$s is equivalent. How to connect them in a cycle so that we can obtain the $mcyc_1$?
	
	We can use the fact that the triangulation is embedded in $\mathbb{R}^2$ to our advantage. We can calculate the orientation of each of the centroids $v_1,v_2, \cdots,v_6$ from the $x$ axis. Suppose the coordinate of the centroids are  $v_i=(x_i,y_i)$, then the orientation $\theta_i=\arctan(\frac{y_i}{x_i})$. We can arrange in the order of increasing orientation angle. For Fig.~\ref{subfig:orientorder} we get $\{\theta_1,\theta_2,\theta_3,\theta_4,\theta_5,\theta_6\}$. We can see that in this case $\phi:2^X \rightarrow \mathbb{R}$ is the function that calculates the orientation of the centroids of the triangles. Now defining a ternary relation $\gamma_\leq$ as in lemma~\ref{lm:app} and then inducing the proximity from the cyclic order we can write $\cnear{\gamma_\leq}(\triangle_1,\triangle_2)=\cnear{\gamma_\leq}(\triangle_2,\triangle_3)=\cnear{\gamma_\leq}(\triangle_3,\triangle_4)=\cnear{\gamma_\leq}(\triangle_4,\triangle_5)=\cnear{\gamma_\leq}(\triangle_5,\triangle_6)=\cnear{\gamma_\leq}(\triangle_6,\triangle_1)=0$. Subsituting $cnt(\triangle_{i})$ for $\triangle_{i}$, this yields a proximity graph that is a cycle $v_1 \rightarrow v_2 \rightarrow v_3 \rightarrow v_4 \rightarrow v_5 \rightarrow v_6 \rightarrow v_1$. This graph is displayed in Fig.~\ref{subfig:mcyc1}.
	
	We must note that by choosing a different $\phi$ we can have cycle ordered in a different way. Another possible choice could be to arrange in the increasing order of area. \qquad \eot
\end{example}

\subsection{Order induced proximities on video frames}\label{subsec:frameord}
We consider approaches to establish an order on video frames, that will lead to the induction of a proximity relation. The first approach considers an order established based on the area of maximal nuclear clusters(MNC) in the triangulated frames. As previously stated, the nerve is a collection of sets(triangles) with a nonempty intersection and a nerve with the most number of sets(triangles) is the MNC.

\begin{algorithm}[!ht]
	\caption{Order Induced Proximity on MNCs in Video Frames based on area}
	\label{alg:mncarea_ord}
	\SetKwData{Left}{left}
	\SetKwData{This}{this}
	\SetKwData{Up}{up}
	\SetKwFunction{Union}{Union}
	\SetKwFunction{FindCompress}{FindCompress}
	\SetKwInOut{Input}{Input}
	\SetKwInOut{Output}{Output}
	\SetKwComment{tcc}{/*}{*/}
	
	\Input{digital video $\mathcal{V}$,Number of keypoints $n$}
	\Output{A set $X=\{\{f_i,mnc_j\}:f_i \in \mathcal{V},mnc_j \subseteq f_i\,and\,i,j \in \mathbb{Z}^+\}$ arranged in increasing order of MNC area}
	\emph{$\mathcal{X} \assign \{\},X \assign \{\}$, declare empty array}\;
	\ForEach{$f_i \in \mathcal{V}$}{
		\emph{$f_i \longmapsto S=\{s_1,\cdots,s_n\}$, where $S$ is the set of keypoints}\;
		\emph{$S \longmapsto \mathcal{T}(S)$, where $\mathcal{T}(S)$ is delaunay triangulation on the keypoints $S$}\;
		\emph{$\mathcal{T}(S) \longmapsto \mathcal{M}$, where $\mathcal{M}=\{mnc_j:j \in \mathbb{Z}^+\}$ is the set of all the MNCs in $\mathcal{T}(S)$}\;
		\ForEach{$mnc_j \in \mathcal{M}$}{
			\emph{$mnc_j \longmapsto a_{ij}$, where $a_{ij}$ is the area of $mnc_j \subset f_i$}\;
			\emph{$\mathcal{X} \assign \{\mathcal{X},\{f_i,mnc_j,a_{ij}\} \}$, appending at the end}\;	
		}
	}
	\emph{$\mathcal{X} \rightarrow \mathcal{X}_{sort}$, where $\mathcal{X}_{sort}$ contains all the $3$-tuples in $\mathcal{X}$ arranged in ascending order of $a_{ij}$}\;
	\ForEach{$\{f_i,mnc_j,a_{ij}\} \in \mathcal{X}_{sort}$}{
		\emph{$\{f_i,mnc_j,a_{ij}\} \longmapsto \{f_i,mnc_j\}$}\;
		\emph{$X \assign \{X,\{f_i,mnc_j\} \}$, appending at the end}\;    
	}
\end{algorithm}

Using these notions let us explain the approach which has been stated in algorithm~\ref{alg:mncarea_ord}. Let $\mathcal{V}$ be the video, that is a collection of frames $\{f_1,\cdots,f_n\}$. For each frame $f_i$ we select keypoints $S=\{s_1, \cdots, s_n\}$ to serve as seeds for Delaunay triangulation $\mathcal{T}(S)$. Once, we have the triangulation we proceed to determining the MNCs that are represented as $\mathcal{M}=\{mnc_1,\cdots,mnc_j\}$. Area of $mnc_j \subset f_i$ is represented as $a_{ij}$. It must be noted that $a_{ij}$ is $\sum \triangle_k$, over all $\triangle_k \in mnc_j \subset f_i$. We have $3$-tuples $\{f_i,mnc_j,a_{ij}\}$ for each frame-MNC pair. These $3$-tuples form the set $\mathcal{X}$, which is then sorted in ascending order for each of area $a_{ij}$ yielding the set $\mathcal{X}_{sort}$. By selecting the first two elements of each $3$-tuple $\{f_i,mnc_j,a_{ij}\} \in \mathcal{X}_{sort}$ we get the corresponding $\{f_i,mnc_j\} \in X$.

For ease of understanding consider that $X=\{x_1,\cdots,x_m\}$, where each $x_i$ is the $2$-tuple that represents a particular frame-MNC pair. Let  $a(x_i)$ be the area of MNC represented by $x_i$. It can be seen that $a(x_i)<=a(x_j)$ for all $j \geq i$ due to the sorting performed in algorithm~\ref{alg:mncarea_ord}. We present some important results regarding this set $X$ sorted with respect to the MNC area.
\begin{lemma}\label{lm:mncarea_ord}
	Let $X=\{x_1,\cdots,x_m\} \, s.t. \, x_1=\{f_i,mnc_j\}$ where $f_i$ is a video frame and $mnc_j \subset f_i$ is a MNC in this frame. Define a function $a:2^{\mathbb{R}^2} \rightarrow \mathbb{R}$, such that $a(x_i)$ is the area of MNC in the frame-MNC pair $x_i$. Then, the inequality relation $\leq$ over the set $A=\{a(x_i): x_i \in X \}$ satisfy the following conditions:
	\begin{compactenum}[1$^o$]
		\item \textbf{Reflexivity:} $a(x_i) \leq a(x_i)$
		\item \textbf{Antisymmetry:} if $a(x_i) \leq a(x_j)$ and $a(x_j) \leq a(x_i)$, then $a(x_i)=a(x_j)$
		\item \textbf{Transitivity:} if $a(x_i) \leq a(x_j)$ and $a(x_j) \leq a(x_k)$, then $a(x_i) \leq a(x_k)$ 
		\item \textbf{Connex property or Totality:} either $a(x_i) \leq a(x_j)$ or $a(x_j) \leq a(x_i)$
	\end{compactenum}
\end{lemma}
\begin{proof}
	\begin{compactenum}[1$^o$]
		\item It follows directly from the fact that area of each MNC is equal to itself.
		\item as area function $a$ outputs real numbers and for two real numbers $a,b$ we know that $(a \leq b) \wedge (b \leq a) \Rightarrow (a=b)$.
		\item the area function output is a real number and for three real numbers $a,b,c$ it is known that $(a \leq b) \wedge (b \leq c) \Rightarrow (a \leq c)$.
		\item the area of an MNC is a real number and for any two real numbers $a,b$ it is known that either $a \leq b$ or $b \leq a$.	
	\end{compactenum}
\end{proof}
From this lemma we can arrive at the following result. 
\begin{theorem}\label{thm:mncarea_ord}
	Let $X=\{x_1,\cdots,x_m\} \, s.t. \, x_1=\{f_i,mnc_j\}$ where $f_i$ is a video frame and $mnc_j \subset f_i$ is a MNC in this frame. Define a function $a:2^{\mathbb{R}^2} \rightarrow \mathbb{R}$, such that $a(x_i)$ is the area of MNC in the frame-MNC pair $x_i$. Then, the inequality relation $\leq$ over the set $A=\{a(x_i): x_i \in X \}$ is a total order.
\end{theorem}
\begin{proof}
	This follows directly from the lemma \ref{lm:mncarea_ord} and  definition \ref{def:torder}.
\end{proof}
As we know that area is a real valued function we can say that areas of $x_i \in X$ form a total order as formed by $(\mathbb{Z},\leq)$ as considered in example~\ref{ex:example1}. Thus, we can also conclude that the Hasse diagram for the areas $a(x_i)$ is similar to the one shown in Fig.~\ref{subfig:integers}. Moreover, this order would induce a proximity on the set $A=\{a(x_i):x_i \in X \}$ as given by def.~\ref{def:toprox}, such that $\tonear{\leq}(a(x_i),a(x_j))=0$ iff $j=i+1$ and $\tonear{\leq}(a(x_i),a(x_j))=1$ otherwise. We can use $\tonear{\leq}$ on $A$ to define a notion of proximity for $X$.  As the set $X$ has frame-MNC pairs arranged in an order of increasing MNC area. We define a proximity relation $\tonear{a_\leq}$ on $X$ such that $\tonear{a_\leq}(x_i,x_j)=0$ iff $j=i+1$ and $\tonear{a_\leq}(x_i,x_j)=0$. Let us explain this proximity on frame-MNC pairs with the help of following example.
\begin{algorithm}[!ht]
	\caption{Extracting subgraphs for a specific frame from the Proximity graph of an order induced proximity on Video frames}
	\label{alg:mncarea_subgraph}
	\SetKwData{Left}{left}
	\SetKwData{This}{this}
	\SetKwData{Up}{up}
	\SetKwFunction{Union}{Union}
	\SetKwFunction{FindCompress}{FindCompress}
	\SetKwInOut{Input}{Input}
	\SetKwInOut{Output}{Output}
	\SetKwComment{tcc}{/*}{*/}
	
	\Input{A set $X=\{\{f_i,mnc_j\}:f_i \in \mathcal{V},mnc_j \subseteq f_i\,and\,i,j \in \mathbb{Z}^+\}$ arranged in some order, frame number under consideration $k$}
	\Output{$G_k$ the subgraph representing order relations for video frame number $k$}
	\emph{$V \assign \{\},E \assign \{\}$, declare empty array}\;
	\ForEach{$\{f_i,mnc_j\} \in X\; s.t.\;i=k$}{
		\emph{$\{f_i,mnc_j\} \longmapsto ind$, the index of $2-$tuple in $X$}\;
		\uIf{$ind=1$}{
			\emph{$V=\{V,X(ind),X(ind+1)\}$ append to the end}\;
			\emph{$E=\{E,de(X(ind),X(ind+1))\}$, $de(a,b)$ is a directed edge from $a$ to $b$}\;
		}
		\uElseIf{$ind=length(X)$}{ 
			\emph{$V=\{V,X(ind-1),X(ind)\}$ }\;
			\emph{$E=\{E,de(X(ind-1),X(ind))\}$}\;
		} 
		\Else{ 
			\emph{$V=\{V,X(ind-1),X(ind),X(ind+1)\}$ }\;
			\emph{$E=\{E,de(X(ind-1),X(ind)),de(X(ind),X(ind+1))\}$}\;
		}
		
	}
	\emph{$G_k \assign (V,E)$, where $V$ is the vertex set and $E$ is the edge set}\;
	
\end{algorithm}
\begin{figure}
	\centering
	\includegraphics[scale=1,bb=125 475 380 665,clip]{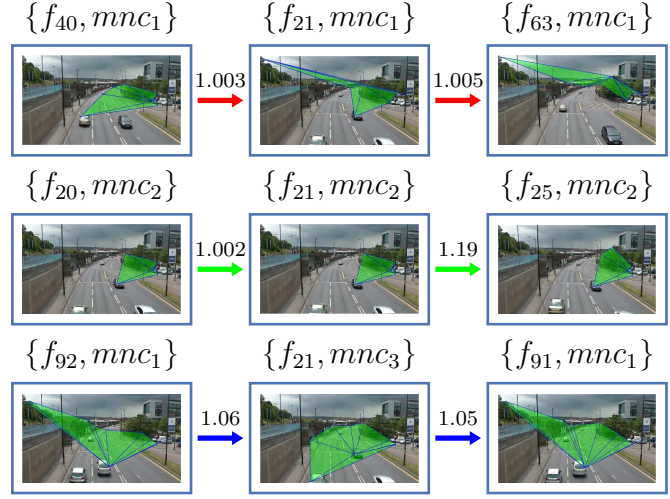}
	\caption{This figure represents the $\tonear{l_\leq}$ for all the frame-MNC pairs,$\{f_{21},mnc_1 \},\{f_{21},mnc_2 \},\{f_{21},mnc_3\}$ in $f_{21}$. This is a proximity relation induced from the total order on the areas of $mnc_j$ in the frames. The arrows are labeled with the weight that scales the area of $mnc$ in the origin frame-MNC pair to the area of $mnc$ in the destination frame}
	\label{fig:framesim_area}
\end{figure}

\begin{example}
	In this example we will use the stock video in MATLAB, "Traffic.mp4" to illustrate how a total order on the areas of the MNCs can be used to induce a proximity on frame-MNC pairs. The process has been detailed in algorithm~\ref{alg:mncarea_ord}. We represent a frame-MNC pair as $(f_i,mnc_j)$ and group them in a set $X=\{x_1,\cdots,x_m\}$. The elements in X are arranged in the increasing order of MNC area, $a(x_i)$. We have seen that $(A,\leq)$, where $A=\{a(mnc_j): mnc_j \in x_i \in X\}$. Def.~\ref{def:toprox} allows a proximity on $A$ such that $\tonear{\leq}(a(x_i),a(x_j))=0$ iff $j=i+1$ and $\tonear{\leq}(a(x_i),a(x_j))=1$ otherwise. This is a direct result of the sorting performed on $X$ in ascending order w.r.t. $a(x_i)$. The condition $(l(x_i) \leq l(x_j)) \wedge (\not\exists l(x_k) \in L \setminus \{l(x_i),l(x_j)\}\; s.t.\;l(x_i)\leq l(x_k) \leq l(x_j))$ in def.~\ref{def:toprox} is reduced to $j=i+1$. Thus, we can write $a(x_1)\leq a(x_2)\leq \cdots \leq a(x_m)$ and induce the proximity relation represented as $\tonear{\leq}(a(x_1),a(x_2))=0 \wedge \tonear{\leq}(a(x_2),a(x_3))=0 \wedge \cdots \wedge \tonear{\leq}(a(x_{m-1}),a(x_m))=0$. Now we can simply lift the proximity from $a(x_i)$ to $x_i$ and simply rewrite this as $\tonear{a_\leq}(x_1,x_2)=0 \wedge \tonear{a_\leq}(x_2,x_3)=0 \wedge \cdots \wedge \tonear{a_\leq}(x_{m-1},x_m)=0$. This gives us a proximity relation on $x_i$, the frame-MNC pairs.
	
	It can be seen from the definition of induced proximity(def.~\ref{def:toprox}) that a frame-MNC pair can at most be related to two other frames, with $x_{i-1}$ is proximal to $x_{i}$ and $x_i$ is proximal to $x_{i+1}$. A frame $f_i$ can have multiple MNCs and thus corresponding number of frame-MNC pairs. All such relations for a particular frame are found using algorithm~\ref{alg:mncarea_subgraph}. This is the implementation of these relations for each of the frame-MNC pairs in a frame. 
	
	We illustrate this using $f_{21}$ from the video having three MNCs and corresponding pairs, $(f_{21},mnc_1),(f_{21},mnc_2),(f_{21},mnc_3)$. The figure~\ref{fig:framesim_area} summarizes the relations for each of these pairs in a path graph corresponding to each of them. The images are labeled with the $(f_i,mnc_j)$ and the edges are labeled with the $\frac{a(x_{i+1})}{a(x_i)}$. It must be noted that $(f_{40},mnc_1)\leq(f_{21},mnc_1)\leq(f_{63},mnc_1)$ and no other frame-MNC pair can be inserted in this chain. MNCs of varying shape can similar areas as is represented by $(f_{40},mnc_1),(f_{21},mnc_1)$ and $(f_{21},mnc_3),(f_{91},mnc_1)$. Moreover, it can be seen that this relation of proximity transcends the temporal order of the frames as $\tonear{a_\leq}((f_{21},mnc_1),(f_{63},mnc_1))=0$.	\qquad \eot
\end{example}

\begin{algorithm}[!ht]
	\caption{Order Induced Proximity on $1$-maximal cycles in Video Frames based on length}
	\label{alg:vortlen_ord}
	\SetKwData{Left}{left}
	\SetKwData{This}{this}
	\SetKwData{Up}{up}
	\SetKwFunction{Union}{Union}
	\SetKwFunction{FindCompress}{FindCompress}
	\SetKwInOut{Input}{Input}
	\SetKwInOut{Output}{Output}
	\SetKwComment{tcc}{/*}{*/}
	
	\Input{digital video $\mathcal{V}$,Number of keypoints $n$}
	\Output{A set $X=\{\{f_i,mnc_j\}:f_i \in \mathcal{V},mnc_j \subseteq f_i\,and\,i,j \in \mathbb{Z}^+\}$ arranged in increasing order of vortex length}
	\emph{$\mathcal{X} \assign \{\},X \assign \{\},C_{sort} \assign \{\}$,  declare empty array}\;
	\ForEach{$f_i \in \mathcal{V}$}{
		\emph{$f_i \longmapsto S=\{s_1,\cdots,s_n\}$, where $S$ is the set of keypoints}\;
		\emph{$S \longmapsto \mathcal{T}(S)$, where $\mathcal{T}(S)$ is delaunay triangulation on the keypoints $S$}\;
		\emph{$\mathcal{T}(S) \longmapsto \mathcal{M}$, where $\mathcal{M}=\{mnc_j:j \in \mathbb{Z}^+\}$ is the set of all the MNCs in $\mathcal{T}(S)$}\;
		\ForEach{$mnc_j \in \mathcal{M}$}{
			\emph{$mnc_j \longmapsto C=\{c_1, \cdots,c_i\}$, where $c_i$ is centroid of $\triangle_i \in mnc_j$}\;
			\emph{$C \longmapsto u$, where $u$ is the centroid of points in $C$}\;
			\emph{$C \longmapsto \Theta=\{\{c_1,\theta_1\},\cdots,\{c_j,\theta_j\}\}$, where $\theta_j=\arctan(\frac{c_{j}(2)-u(2)}{c_{j}(1)-u(1)})$}\;
			\emph{$\Theta \longmapsto \Theta_{sort}$, where $\Theta_{sort}$ is arranged in order of increasing $\theta$}\;
			\ForEach{$\{c_k,\theta_k \} \in \Theta_{sort}$}{
				\emph{$\{c_k, \theta_k\} \longmapsto c_k$}\;
				\emph{$C_{sort} \assign \{C_{sort},c_k\}$}\;
			}
			\emph{$\Theta_{sort} \longmapsto C_{sort}$, project onto $1st$ cordinate}\;
			\emph{$C_{sort}=\{cs_1,\cdots,cs_j\} \longmapsto v_j=cyc(cs_1,\cdots,cs_j)$}\;
			\emph{$v_j \longmapsto l_{ij}$, $l_{ij}$ is the length of $v_j$}
			\emph{$\mathcal{X} \assign \{\mathcal{X},\{f_i,v_j,l_{ij}\} \}$, appending at the end}\;	
		}
	}
	\emph{$\mathcal{X} \rightarrow \mathcal{X}_{sort}$, where $\mathcal{X}_{sort}$ contains all the $3$-tuples in $\mathcal{X}$ arranged in ascending order of $l_{ij}$}\;
	\ForEach{$\{f_i,v_j,l_{ij}\} \in \mathcal{X}_{sort}$}{
		\emph{$\{f_i,v_j,l_{ij}\} \longmapsto \{f_i,v_j\}$}\;
		\emph{$X \assign \{X,\{f_i,v_j\} \}$, appending at the end}\;    
	}
\end{algorithm}
Another method to induce a proximity relation is to use the length of $1$-maximal cycle $mcyc_1$ that is the cycle constructed using the centroids of triangles in the MNC. A method for constructing such cycles has been discussed in section~\ref{subsec:mcv}. The method has been detailed in algorithm~\ref{alg:vortlen_ord}. For each frame $f_i \in \mathcal{V}$, we select keepoints $S=\{s_1,\cdots,s_n\}$ upon which the Delaunay triangulation $\mathcal{T}(S)$ is constructed. We determine $\mathcal{M}=\{mnc_1,\cdots,mnc_j\}$, the set of MNCs in $f_i$. For each $mnc_j \in \mathcal{M}$ we determine the centroids of triangles constituting it,$C=\{c_1,\cdots,c_i\}$.

We calculate the centroid of points in $C$ denoting it as $u$, and then using this as the origin we calculate orientation $\theta_i$ of the position vectors of $c_i \in C$. Now connect these points in the increasing order of $\theta_i$ and close the loop to get $mcyc_1$. The $mcyc_1$ for $mnc_j$ is denoted $v_j$. The length of $v_j \subset f_i$ is $l_{ij}$. We form an array of $3$-tuples $\mathcal{X}=\{\{f_i,v_j,l_{ij}\}: f_i \in \mathcal{V},v_j \subset mnc_j\}$. This array is sorted in the ascending order of $l_{ij}$ to yield $X_{sort}$. Projection of the $3$-tuple $\{f_i,v_j,l_{i,j}\}$ onto the first two elements gives us the frame-cycle pair $\{f_i,v_j \}$. The set $X$ is the collection of all such pairs arranged in increasing order of length $l_{ij}$.

We express this set $X$ as $\{x_1,\cdots,x_m\}$ where each $x_m=\{f_i,v_j\}$ is a frame-cycle pair. Let $l(x_i)$ be the length of $_j$, the cycle represented by $x_i$. Then we can see that due to the sorting performed $l(x_i) \leq l(x_j)$ for all $j \geq i$. We present the following results.
\begin{lemma}\label{lm:vortlen_ord}
	Let $X=\{x_1,\cdots,x_m\} \, s.t. \, x_1=\{f_i,v_j\}$ where $f_i$ is a video frame and $v_j \subset f_i$ is a $mcyc_1$ in this frame. Define a function $l:2^{\mathbb{R}^2} \rightarrow \mathbb{R}$, such that $l(x_i)$ is the length of $mcyc_1$ in the frame-cycle pair $x_i$. Then, the inequality relation $\leq$ over the set $L=\{l(x_i): x_i \in X \}$ satisfy the following conditions:
	\begin{compactenum}[1$^o$]
		\item \textbf{Reflexivity:} $l(x_i) \leq l(x_i)$
		\item \textbf{Antisymmetry:} if $l(x_i) \leq l(x_j)$ and $l(x_j) \leq l(x_i)$, then $l(x_i)=l(x_j)$
		\item \textbf{Transitivity:} if $l(x_i) \leq l(x_j)$ and $l(x_j) \leq l(x_k)$, then $l(x_i) \leq l(x_k)$ 
		\item \textbf{Connex property or Totality:} either $l(x_i) \leq l(x_j)$ or $l(x_j) \leq l(x_i)$
	\end{compactenum}
\end{lemma}
\begin{proof}
	\begin{compactenum}[1$^o$]
		\item It follows directly from the fact that length of each $mcyc_1$ is equal to itself.
		\item as the range of $l$ is the set $\mathbb{R}$ and for two real numbers $a,b$ we know that $(a \leq b) \wedge (b \leq a) \Rightarrow (a=b)$.
		\item the length function output is a real number and for three real numbers $a,b,c$ it is known that $(a \leq b) \wedge (b \leq c) \Rightarrow (a \leq c)$.
		\item the length is a real number and for any two real numbers $a,b$ it is known that either $a \leq b$ or $b \leq a$.	
	\end{compactenum}
\end{proof}
This lemma leads to the following theorem. 
\begin{theorem}\label{thm:vortlen_ord}
	Let $X=\{x_1,\cdots,x_m\} \, s.t. \, x_1=\{f_i,mnc_j\}$ where $f_i$ is a video frame and $v_j \subset f_i$ is a $mcyc_1$ in this frame. Define a function $l:2^{\mathbb{R}^2} \rightarrow \mathbb{R}$, such that $l(x_i)$ is the length of $mcyc_1$ in the frame-cycle pair $x_i$. Then, the inequality relation $\leq$ over the set $L=\{l(x_i): x_i \in X \}$ is a total order.
\end{theorem}
\begin{proof}
	This follows directly from the lemma \ref{lm:vortlen_ord} and  definition \ref{def:torder}.
\end{proof}
Thus, similar to the case with MNC area we can see that length of $mcyc_1$ is areal number and the order on lengths of such cycles is a total order simialr to $(\mathbb{Z},\leq)$ considered in example \ref{ex:example1}. Accordingly the Hasse diagram is similar in structure to Fig.~\ref{subfig:integers}. Thus, this order induces a proximity on $L=\{l(x_i): x_i \in X \}$ as per definition~\ref{def:toprox}. It can be seen that $\tonear{\leq}(l(x_i),l(x_j))=0$ iff $j=i+1$ and $\tonear{\leq}(l(x_i),l(x_j))=1$ otherwise. This proximity on $L$ can be extended to a similar notion on set $X$. The set $X$ as we know from algorithm~\ref{alg:vortlen_ord} is the list of frame-cycle pairs arranged in the increasing order of length. Thus, we define $\tonear{l_\leq}$ a proximity relation in X such that $\tonear{l_\leq}(x_i,x_j)=0$ iff $j=i+1$ and $\tonear{l_\leq}(x_i,x_j)=1$ otherwise. We explain this with the following example.
\begin{figure}
	\centering
	\includegraphics[scale=1,bb=125 475 380 665,clip]{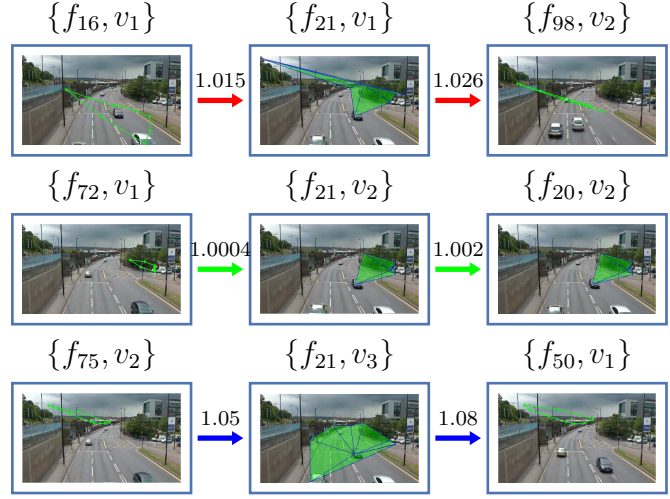}
	\caption{This figure represents the $\tonear{l_\leq}$ for all the frame cycle pairs,$\{f_{21},v_1 \},\{f_{21},v_2 \},\{f_{21},v_3\}$ in $f_{21}$. This is a proximity relation induced from the total order on the lengths of $mcyc_1$ in the frames. The arrows are labelled with the weight that scales the length of $mcyc_1$ on the origin frame-cycle pair to the length of $mcyc_1$ in the destination frame-cycle pair.}
	\label{fig:framesim_length}
\end{figure}
\begin{example} 
	In this example we consider how a total order(def.~\ref{def:torder}) on the length of $mcyc_1$ induces a proximity $\tonear{l_\leq}$ on the video frames. We use "Traffic.mp4" which is a stock video in MATLAB. We tesselate each frame $f_i$ and construct $v_j$ for each MNC $mnc_j \subset f_i$. It must be noted that we consider frame-cycle pairs as an entity i.e. if a frame has multiple MNC leading to multiple $mcyc_1$ each pair will be considered separately. As we have seen that $X=\{x_1,\cdots,x_m\}$ where $x_m=\{f_i,v_j\}$ is a frame-cycle pair is sorted with increasing length of $v_j$ represented as $l(v_j)$. $(L,\leq)$ where $L=\{l(v_i):v_i \in x_i \in X\}$, is a totaly orderd set. Using the def.~\ref{def:toprox} we can easily determine that $\tonear{\leq}(l(x_i),l(x_j))=0$ iff $j=i+1$ and $\tonear{\leq}(l(x_i),l(x_j))=1$ otherwise. This happens because we have sorted $X$ in ascendimg order w.r.t to $l(x_i)$, for $x_i \in X$. Thus the condition $(l(x_i) \leq l(x_j)) \wedge (\not\exists l(x_k) \in L \setminus \{l(x_i),l(x_j)\}\; s.t.\;l(x_i)\leq l(x_k) \leq l(x_j))$ in def.~\ref{def:toprox} is reduced to $j=i+1$. We just lift this proximity relation on the lengths $l(x_i)$ to the frame-cycle pairs $x_i$ and term it $\tonear{l_\leq}$. As discussed this is defined as $\tonear{l_\leq}(x_i,x_j)=0$ iff $j=i+1$ and $\tonear{l_\leq}(x_i,x_j)=1$ otherwise. Thus, as we can write $l(x_1) \leq l(x_2) \leq \cdots \leq l(x_m)$ we can write the induced proximites as $\tonear{\leq}(l(x_1),l(x_2))=0 \wedge \tonear{\leq}(l(x_2),l(x_3))=0 \wedge \cdots \wedge \tonear{\leq}(l(x_{m-1}),l(x_m))=0$.  By lifting the proximity from $l(x_i)$ to $x_i$ we have $\tonear{l_\leq}$. For which we can write $\tonear{l_\leq}(x_1,x_2)=0 \wedge \tonear{l_\leq}(x_2,x_3)=0 \wedge \cdots \wedge \tonear{l_\leq}(x_{m-1},x_m)=0$. Thus, we have a proximity relation on the frame-cycle pairs.
	
	We can see that a frame-cycle $x_i$ pair can be related to two other pairs, with $x_{i-1}$ being proximal to $x_i$ and $x_i$ being proximal to $x_{i+1}$. So for a particular frame we can have multiple frame-cycle pairs due to multiple MNC. We determine all such relations for frame-cycle pairs of a particular $f_i$. This is done using algorithm~\ref{alg:mncarea_subgraph} which is an implementation of the explained relations. To illustrate our point we choose a frame $f_{21}$ having three $mcyc_1$ represented as $\{f_{21},v_1 \},\{f_{21},v_2 \},\{f_{21},v_3\}$. In fig.~\ref{fig:framesim_length} we represent the proximity relations for each of these pairs. We have a path graph for each of the three frame-cycle pairs. Each of the images is labelled with the corresponding $(f_i,v_j)$. The ratio $\frac{l(x_{i+1})}{l(x_i)}$ is labelled as a weight on each of the arrows. It must be noted that $l(\{f_{16},v_1\}) \leq l(\{f_{21},v_1\}) \leq l(\{f_{98},v_2\}) $ and no other frame-cycle pair can be inserted in between the terms in this inequality chain. Moreover, we see that $mcyc_1$ of very different shapes can have similar lengths such as $\{f_{98},v_2\},\{f_{21},v_1\}$ and $\{f_{75},v_2\},\{f_{21},v_3\}$. An important thing to note is that this proximity relates frames that may not be temporally adjacent such as $f_{98}$ and $f_{21}$.  \qquad \eot
\end{example} 

\section{Conclusions}\label{sec:conclusions}
This paper defines and studies the proximity relations induced by partial, total and cyclic orders.The equivalence between the Hasse diagram of the orderd set and the proximity graph of the induced relation has been established. Moreover, an application of this induced proximity to constructing maximal centroidal vortices has been presented. Another application with regards to inducing a proximity relation on video frames has also been presented.
\bibliographystyle{amsplain}
\bibliography{orderrefs}
\end{document}